\documentclass[12pt]{extarticle}
\usepackage{amsmath, amsthm, amssymb, hyperref, color}
\usepackage[shortlabels]{enumitem}
\usepackage{graphicx}
\usepackage[all]{xypic}
\usepackage{makecell}

\oddsidemargin \evensidemargin
\newtheorem{theorem}{Theorem}
\numberwithin{theorem}{section}
\newtheorem{proposition}[theorem]{Proposition}
\newtheorem{lemma}[theorem]{Lemma}
\newtheorem{corollary}[theorem]{Corollary}

\newtheorem{remark}[theorem]{Remark}
\newtheorem{example}[theorem]{Example}

\newcommand{\RR}{\mathbb{R}}
\newcommand{\QQ}{\mathbb{Q}}
\newcommand{\PP}{\mathbb{P}}
\newcommand{\CC}{\mathbb{C}}

 \date{}
 
\title{\textbf{Duality of Multiple Root Loci}}

\author{Hwangrae Lee and Bernd Sturmfels}

\begin{document}

\maketitle

\begin{abstract} \noindent
The multiple root loci among univariate polynomials of degree~$n$ are indexed 
by partitions of $n$. We study these loci and their conormal varieties. 
The projectively dual varieties are joins of such loci
where the partitions are hooks. Our emphasis 
lies on equations and parametrizations that are useful for Euclidean distance optimization.
We compute the ED degrees for hooks.
Among the dual hypersurfaces are those that
demarcate the set of binary forms whose real rank equals the generic complex rank.
\end{abstract}

\section{Introduction}

Univariate polynomials of degree $n$ correspond to points in a projective space $\PP^n$.
The {\em multiple root locus} $\Delta_\lambda$ associated with a partition $\lambda = 
(\lambda_1,\ldots,\lambda_d)$ is the subvariety of $\PP^n$ given by
 polynomials that have $d$
distinct roots with multiplicities $\lambda_1,\ldots,\lambda_d$.
The dimension of $\Delta_\lambda$ is $d$.
The singular locus of $\Delta_\lambda$ is the union of
 certain codimension one subloci $\Delta_\mu$, as described in \cite[\S 3]{Kur}.
 The degree of $\Delta_\lambda$ was determined by
Hilbert in \cite{Hil}. He showed that
\begin{equation}
\label{eq:hilbert}
 {\rm deg}( \Delta_\lambda) \,\, = \,\,
\frac{d !} {m_1 ! m_2 ! \cdots m_p !}
\cdot \lambda_1 \lambda_2 \cdots \lambda_d,
\end{equation}
where $m_j$ denotes the number of parts $\lambda_i$ in the partition $\lambda$
that are equal to the  integer $j$.

The multiple root loci $\Delta_\lambda$ 
have been studied in a wide range of contexts and
under various different names:
coincident root loci \cite{Chi1, FNR, Kur},
 pejorative manifolds \cite{McN, ZZ1},
strata of the discriminant \cite{Hil, Kat, Na}, 
$\lambda$-Chow varieties \cite{Oed},
factorization manifolds \cite[Definition 5.2.4]{ZZ2}, etc.
Our motivation arose from the desire to understand the geometry of a  model selection problem
considered at the  interface of  symbolic computation \cite{KMYZ} and
numerical analysis \cite{ZZ1}:
given a univariate polynomial $h$, identify a 
low-dimensional $\Delta_\lambda$ such that
$h$ is close to~$\Delta_\lambda$.

Finding a point in $\Delta_\lambda$ that is closest to a 
given $h$ is a problem of polynomial optimization \cite{BPT}.
We here characterize the
geometric duality that underlies this optimization problem,
in the sense of \cite[Chapter 5]{BPT}.
The  key player is the dual variety $(\Delta_\lambda)^\vee$.
This variety lives in the dual projective space $\PP^n$ and it parametrizes
all hyperplanes that are tangent to $\Delta_\lambda$.

The duals to multiple root loci were studied
by Oeding in \cite{Oed}. He shows that
$(\Delta_\lambda)^\vee$ is a hypersurface
if and only if $m_1 =0$, i.e.~all parts of $\lambda$ satisfy $\lambda_i \geq 2$.
In \cite[Theorem 5.3]{Oed} an explicit formula
is given for the degree of the polynomial 
that cuts out this hypersurface:
\begin{equation}
\label{eq:oeding}
{\rm deg}\bigl((\Delta_\lambda)^\vee \bigr) \,\, = \,\,
 \frac{(d+1) !}{m_2 !  \cdots m_p!}\cdot (\lambda_1-1)(\lambda_2 - 1) \cdots (\lambda_d-1).
\end{equation}
For the application to optimization  in \cite[Theorem 5.23]{BPT},
this is the number of complex critical points one encounters when minimizing a 
general linear function over an affine chart of $\Delta_\lambda$.
For instance, consider $n  = 5$ and $\lambda  = (3,2)$.
Following Example \ref{ex:apple} and \cite[\S 3]{CO},
the polynomial for  $(\Delta_\lambda)^\vee $ is the 
{\em apple invariant} of degree $12$.
So, optimizing a linear function over quintics
with a triple root and a double root leads to solving an equation of degree $12$.

The present paper is a continuation of the studies by Hilbert \cite{Hil}
and Oeding \cite{Oed}.  It is organized as follows.
In Section 2 we set up notation and basics.
In Theorem \ref{main2}  we parametrize
the conormal variety ${\rm Con}_\lambda$
 that links $\Delta_\lambda$ and $(\Delta_\lambda)^\vee$.
Theorem \ref{thm:neun} offers a parametrization for
the projective duals of multiple root loci.
 These results are derived from the
   apolarity theory for binary forms,
as described in the book by Iarrobino and Kanev \cite{IK}.

In Section 3 we study the multidegree of
the conormal variety ${\rm Con}_\lambda$, and we summarize what is known about the
ideals of $\Delta_\lambda$ and $(\Delta_\lambda)^\vee$.
 Table~\ref{tab:eins} offers a census of  small instances up to $n = 7$.
We also examine duality for hooks, and inclusions among the dual strata.

Section 4 offers an application to tensor rank over $\RR$.
Theorem \ref{thm:realrank} characterizes the algebraic boundary
of the set of binary forms whose real rank equals the generic complex rank.
When $n$ is even then this involves the use of
Chow forms \cite{GKZ} and Hurwitz forms \cite{Stu}.

In Section 5 we turn to Euclidean distance (ED) optimization.
Theorem \ref{thm:itisn} determines the ED degree of $\Delta_\lambda$
for hook shapes $\lambda$, both for
invariant coordinates and for generic coordinates.
This generalizes known results  
for the rational normal curve $\Delta_{(n)}$, in particular the
tight connection to eigenvectors of
symmetric tensors; cf.~\cite[Ex.~5.2 and Cor.~8.7]{DHOST}.

Section 6 discusses implications for the ED optimization
problem (\ref{eq:minimize}).
The focus lies on locating all the critical points
in $\Delta_\lambda$ and in $(\Delta_\lambda)^\vee$,
and on certifying the global optimum.
We work this out for several examples,
including one from tensor decomposition over~$\RR$.

\section{Duality for Binary Forms}

We fix a field $K$ of characteristic zero and we represent
univariate polynomials by binary forms.
Let $V = K[x,y]_n$ denote the space of binary forms of degree $n$
and $V^\vee = K[u,v]_n$ its dual vector space. For $f \in V$ and $
g \in V^\vee$, the pairing is defined by using partial derivatives:
\begin{equation}
\label{eq:innerprod1}
 \langle f,g \rangle\, \,\, = \,\, \,
\frac{1}{n!} \cdot g\left(\frac{\partial}{\partial x},\frac{\partial}{\partial y} \right)
 f(x,y) \,\,\,\, = \,\,\,\,
 \frac{1}{n!} \cdot
 f \left(\frac{\partial}{\partial u},\frac{\partial}{\partial v} \right)  g(u,v).
\end{equation}
This is the scalar in $K$ obtained by interpreting one polynomial as
a differential operator and applying it to the other polynomial. 
Introducing coordinates on $V$ and $V^\vee$, we have:
\begin{equation}
\label{eq:innerprod2}
\hbox{If $\,f=\sum_{i=0}^n \binom{n}{i}a_i x^i y^{n-i}\,$ and
 $\,g = \sum_{i=0}^n \binom{n}{i}b_i u^i v^{n-i}\,\,$
then $\,\,\langle f, g\rangle \,= \, \sum_{i=0}^n \binom{n}{i} a_i b_i$.}
\end{equation}
We regard $f$ and $g$ as elements in the projective spaces
$\PP(V)$ and $\PP(V^\vee)$. Both of these spaces
are identified with $\PP^n$ using homogeneous coordinates
$(a_0:\cdots:a_n)$ and $(b_0:\cdots:b_n)$.

A partition $\lambda$ of $n$ is represented
alternatively as a list of integers $(\lambda_1,\ldots,\lambda_d)$
satisfying $\lambda_1 \geq \cdots \geq \lambda_d \geq 1$
and $\sum_{i=1}^d \lambda_i = n$, or as a multiset
$\{1^{m_1}, \dots, p^{m_p}\}$ satisfying $\sum_{j=1}^p j m_j = n$.
So, $m_j  = \# \{i : \lambda_i = j\}$, and
$d$ is the number of parts of $\lambda$,
and $p$ is the largest part of $\lambda$.

The multiple root variety $\Delta_\lambda \subset \PP^n$  comprises
all binary forms $f$ of degree $n$ that have 
$m_j$ roots of multiplicity $j$.
Equivalently, writing $\ell_i$ for linear forms,
$\Delta_\lambda$  is the image of
$$ (\PP^1)^d \,\rightarrow\, \PP^n ,\,\,
 (\ell_1,\ell_2 ,\ldots, \ell_d) \,\mapsto\,
 f = \ell_1^{\lambda_1}  \ell_2^{\lambda_2} \cdots \ell_d^{\lambda_d}. $$
The variety $\Delta_\lambda$ has dimension $ d$ and degree as in (\ref{eq:hilbert}). 
Among its smooth points are those given
by $d$ distinct linear forms $\ell_i$.
We determine the tangent space of $\Delta_\lambda$
at such a point.

\begin{lemma}
\label{lem:eins}
The tangent space of $\Delta_\lambda$ at a general smooth point
 $f = \prod_{i=1}^d \ell_i^{\lambda_i}$ equals
$$ T_f \Delta_{\lambda} \,\,= \,\, \left\{ 
  h(x,y) \cdot \prod_{i=1}^d \ell_i^{\lambda_i-1}(x,y)
  \,\, \bigg| \, \,h \in \PP( K[x,y]_d) 
  \right\} \,\,\, \simeq \,\,\, \PP^d . $$
\end{lemma}

\begin{proof}
Write $\ell_i = a_ix + b_iy$ for some indeterminate $(a_i:b_i) \in \PP^1$.
The tangent space $T_f \Delta_{\lambda}$ is spanned by the binary forms
$\,\frac{\partial}{\partial a_i} f =x f/\ell_i \,$ 
and $\,\frac{\partial}{\partial b_i}f = y f/\ell_i\,$
for $i=1,2,\ldots,d$. These generators have the required form
if we take $h(x,y)$ to be $\, x {\cdot} (\prod_{j=1}^d \ell_j)/\ell_i\,$
or $\, y {\cdot} (\prod_{j=1}^d \ell_j)/\ell_i$.
For the converse we note that
$K[x,y]_{d-1}$ is spanned by 
$\big\{ (\prod_{j=1}^d \ell_j)/\ell_i\,| \, i = 1,\ldots,d \bigr\}$.
Indeed, this is a Lagrange basis, since the $\ell_i$
are pairwise linearly independent. Hence, by multiplying these
polynomials with $x$ and $y$, we obtain a spanning
set for the vector space $K[x,y]_d$.
\end{proof}

\begin{lemma}
\label{lem:zwei}
Let $f$ be as above and $g \in \PP(V^\vee )$. Then $\,g \perp T_f \Delta_{\lambda}$ if and only if 
\begin{equation}
\label{eq:annihilates}
\left[\prod_{i=1}^d \ell_i^{\lambda_i-1}\bigl(\frac{\partial}{\partial u},
\frac{\partial}{\partial v}\bigr)\right] \,\,\, {\rm annihilates} \,\,\, g(u,v).
\end{equation}
\end{lemma}

\begin{proof}
Let $\tilde g(u,v)$ be the binary form of degree $d$
obtained from $g(u,v)$ by applying the operator on the left
of (\ref{eq:annihilates}). Then
$\tilde g$ is zero if and only if $\tilde g(u,v)$ is
annihilated by $h(\frac{\partial}{\partial u}, \frac{\partial}{\partial v})$
for all $h \in K[x,y]_d$ if and only if $g(u,v)$ is annihilated by
$\bigl(h \cdot \prod_{i=1}^d \ell_i^{\lambda_i-1}\bigr)(\frac{\partial}{\partial u}, 
\frac{\partial}{\partial v})$ for all $h \in K[x,y]_d$. By Lemma \ref{lem:eins},
this means that $g(u,v)$ is orthogonal to 
the  space $T_f \Delta_\lambda$.
\end{proof}

The {\em conormal variety} of 
$\Delta_\lambda$, here denoted ${\rm Con}_\lambda$,
is the Zariski closure 
of the set
$$ \bigl\{ (f,g) \,\, \big| \,\,
f \,\, \hbox{is a smooth point of} \,\,\Delta_\lambda\,\,
{\rm and} \,\,g \perp T_f \Delta_\lambda \bigr\} 
\quad
\hbox{in} \quad \PP(V) \times \PP(V^\vee) \, = \,\PP^n \times \PP^n.
$$
General results on projective duality ensure that
${\rm Con}_\lambda$ is an irreducible
variety of dimension $n{-}1$ in $\PP^n \times \PP^n$.
The {\em dual variety} $(\Delta_\lambda)^\vee$ is
the image of the conormal variety ${\rm Con}_\lambda$ under
 projection onto the second factor  $\PP(V^\vee)$.
This is an irreducible variety of dimension $\leq n-1$.
The Biduality Theorem  \cite{GKZ} states that
 the conormal variety of $(\Delta_\lambda)^\vee$
coincides with ${\rm Con}_\lambda$, and hence
$((\Delta_\lambda)^\vee)^\vee = \Delta_\lambda$.
Lemma~\ref{lem:zwei} implies the following description of the dual variety.

\begin{corollary}\label{main1}
The points on the variety $(\Delta_\lambda)^\vee$
are binary forms $g(u,v)$ that are
annihilated by some order $n-d$ operator of the form
$\,\prod_{i=1}^d \ell_i^{\lambda_i-1}\bigl(\frac{\partial}{\partial u},
\frac{\partial}{\partial v}\bigr)\,$ where
$\,\ell_1,\ldots,\ell_d \in K[x,y]_1$.
\end{corollary}

At this point, let us pause to illustrate the concepts seen above with a small example.

\begin{example} \rm
Let $n = 3$ and take $\lambda = (3)$,
the partition with a single part. 
The conormal variety ${\rm Con}_{(3)}$
is the surface in $\PP^2 \times \PP^2$
whose defining homogeneous prime ideal equals
\begin{equation}
\label{eq:twistedcubic} \!\!\!
\begin{matrix}
\bigl\langle \,a_2^2-a_1 a_3, a_1 a_2-a_0 a_3, a_1^2-a_0 a_2 \,,\,
3 b_1^2 b_2^2 - 4 b_0 b_2^3 - 4 b_1^3 b_3+6 b_0 b_1 b_2 b_3-b_0^2 b_3^2
\qquad \qquad \\ \qquad
a_0 b_0+2 a_1 b_1+a_2 b_2,\,
a_0 b_1+2 a_1 b_2+a_2 b_3,\, a_1 b_0+2 a_2 b_1+a_3 b_2,\,
a_1 b_1+2 a_2 b_2+a_3 b_3 \, \bigr\rangle.
 \end{matrix}
 \end{equation}
 The multiple root locus $\Delta_{(3)}$  is the twisted cubic
 curve in $\PP^3$, consisting of cubes of linear forms,
and defined by the first three quadrics in (\ref{eq:twistedcubic}).
Its dual is $(\Delta_{(3)})^\vee = \Delta_{(2,1)}$. It consists
of binary cubics with a double root, and it is the discriminant surface
in $\PP^3$ defined by the quartic in (\ref{eq:twistedcubic}).
We have ${\rm Con}_{(3)} ={\rm Con}_{(2,1)}$,
after swapping the $a$-variables with the $b$-variables.

We illustrate Corollary \ref{main1} for the point
$\bigl((1:1:1:1), (1:2:-5:8)\bigr)$ in ${\rm Con}_{(3)}$.
This  represents the pair $(f,g)$ where
$f = (x+y)^3$ and $g = (u-v)^2 (u+8v)$.
Then $g(u,v)$ lies in $(\Delta_{(3)})^\vee$ because
$(\frac{\partial}{\partial u} + \frac{\partial}{\partial v})^2 g(u,v) = 0$,
and $f(x,y)$ lies in $ (\Delta_{(2,1)})^\vee$ because
$(\frac{\partial}{\partial x} - \frac{\partial}{\partial y}) f(x,y) = 0$.
Note that the degree formulas (\ref{eq:hilbert}) and (\ref{eq:oeding})
evaluate to $3$ and $4$ for $\lambda = (3)$.
\hfill $\diamondsuit$
\end{example}

The main result in this section is
a parametric representation of the conormal variety.

\begin{theorem}\label{main2}
The conormal variety ${\rm Con}_\lambda$ is the set of pairs
$(f,g) \in \PP^n \times \PP^n$ of the form
\begin{equation}
\label{eq:f_and_g} 
 f(x,y) \,=\, \prod_{i=1}^d (t_i x - s_i y)^{\lambda_i}  \quad \hbox{and} \quad
g(u,v) \, = \sum_{i=1 \atop \lambda_i \neq 1}^d (s_i u + t_i v)^{n-\lambda_i+2} \cdot g_i(u,v), 
\end{equation}
where $(s_i:t_i)$ runs over $\PP^1$ and $g_i$ runs over binary forms of degree $\lambda_i-2$.
This parametrization is a finite-to-one map
$\,  (\PP^1)^{m_1} \times \mathcal{J}_\lambda \dashrightarrow 
{\rm Con}_\lambda \subset \PP^n \times \PP^n \,$
whose degree is $m_1 ! m_2 ! \cdots m_p !$.
\end{theorem}

In order for this theorem to make sense, we need to define
the parameter space. We set
$$ \mathcal{J}_\lambda \,\, = \,\,
{\rm Join} \bigl( \PP^1 \times \PP^{\lambda_1-2}, \,
\PP^1 \times \PP^{\lambda_2-2},\,\ldots,\,
 \PP^1 \times \PP^{\lambda_d-2}\bigr). $$
This is the {\em free join} (or {\em abstract join}) of
the $d-m_1$ varieties in the argument.
Here $m_1$ is subtracted from $d$ because 
the $j$-th argument is empty and gets deleted when $\lambda_j =1$.
Note that $\mathcal{J}_\lambda$ is the projective toric variety whose associated
polytope is the free join of
the product of simplices $\sigma_1 \times \sigma_{{\lambda_j}-2}$
for $j=1,2,\ldots,d$. The dimension of this join equals
$$ {\rm dim}(\mathcal{J}_\lambda) \,\,=\,\, (d-m_1-1) \,+\,
 \sum_{i=1}^d {\rm max}(0,1+(\lambda_i-2))\,\, = \,\, n - m_1 - 1. $$
 The point of the toric variety $\mathcal{J}_\lambda$ is to
 ensure that (\ref{eq:f_and_g}) gives a well-defined rational map.
 We will see in   Example \ref{ex:notmorphism} that it is not a morphism.
 But that is not a problem since we can replace both
 the parameter space and image by their affine cones.  The formulas
   in (\ref{eq:f_and_g}) are homogeneous. We use them in the next section
   to carry out the computations for Table~\ref{tab:eins}.

By projection onto the second factor,
Theorem \ref{main2} immediately yields a parametrization 
$ (\PP^1)^{m_1} \times \mathcal{J}_\lambda \dashrightarrow  (\Delta_\lambda)^\vee \subset \PP^n \,$
of the dual variety. Indeed, $(\Delta_\lambda)^\vee$
consists of all polynomials $g(u,v)$ of the form in (\ref{eq:f_and_g}).
This can be rephrased in the language of projective geometry:

\begin{corollary}\label{main3}
The dual variety $(\Delta_{\lambda})^{\vee} $ 
indexed by $\lambda$ is the
 join of $d-m_1$ multiple root loci determined by hook shapes,
namely the loci $ \Delta_{\{1^{\lambda_i-2}\!,\,n-\lambda_i+2\}}$,
where $i=1,\ldots,d$ with $\lambda_i \geq 2$.
\end{corollary}

Recall that a partition $\lambda$ is a {\em hook} if
at most one part is not $1$. This is special for us:

\begin{corollary} \label{cor:selfdual}
The dual variety $(\Delta_\lambda)^\vee$
is also a multiple root locus $\Delta_\mu$ if and only if the partition
 $\lambda$ is a hook. In that case
$ \lambda = \{1^{n-a},\,a\}$ and  $\mu =  \{1^{a-2},\,n-a+2\}$
for some $a \in \{ 2,\ldots,n\}$. In particular,
 $\Delta_\lambda$ is self-dual whenever
$n$ is even and
$\lambda = \{1^{n/2-1},\,n/2+1\}$.
\end{corollary}

We easily derive the theorem from results that
are well-known in commutative algebra.

\begin{proof}[Proof of Theorem~\ref{main2}]
We use apolarity as presented in the book
of Iarrobino and Kanev \cite{IK}. Let $f$ and $g$
be binary forms, possibly of different degrees.
We say that $f$ is {\em apolar} to $g$
if $g(x,y)$ is annihilated by the operator $f(\frac{\partial}{\partial x}, \frac{\partial}{\partial y})$.
Our  Lemma \ref{lem:zwei} says that
the tangent space $T_f \Delta_\lambda$ consists
of all binary forms $g$ of degree $n$ such that
$\ell_1^{\lambda_1-1}
\ell_2^{\lambda_2-1} \cdots \ell_d^{\lambda_d-1}$ is apolar to $g$.
By \cite[Lemma 1.31]{IK}, this condition is equivalent to 
$g$ having a  {\em generalized additive decomposition} of the form
$\,g =  \ell_1^{n-(\lambda_1-1)+1} g_1 + 
\cdots + \ell_d^{n-(\lambda_d-1)+1} g_d$. 
Here, only terms
with $\lambda_i \geq 2$ may appear.
This is precisely the representation on the right of 
(\ref{eq:f_and_g}), and we conclude that the
proposed parametric representation of the conormal variety ${\rm Con}_\lambda$ is correct.

The parametrization is finite-to-one because ${\rm dim}({\rm Con}_\lambda) = n-1$,
which is the dimension of the parameter space $  (\PP^1)^{m_1} \times \mathcal{J}_\lambda$.
  Consider the fiber over a general point  $(f,g)$ in $ {\rm Con}_\lambda$.
  The first entry $f$ is the image of  precisely $m_1!m_2!\cdots m_p!$ 
  points  in $(\PP^1)^d$.  For each such point
  $\bigl((s_1:t_1),\ldots,(s_d:t_d)\bigr)$, the second entry
  gives a homogeneous system of linear equations:
 \begin{equation}
 \label{eq:dualpara2}
  {\rm const.} \cdot g(u,v) \,\, = \,\, 
\sum_{i=1 \atop \lambda_i \neq 1}^d (s_i u + t_i v)^{n-\lambda_i+2} \cdot g_i(u,v). 
\end{equation}
The unknowns are the coefficients of $g_1,\ldots,g_d$. The
solution set is a linear subspace of $\mathcal{J}_\lambda$.
It is non-empty since $(f,g)$ was assumed to lie in
$ {\rm Con}_\lambda$. Since the parametrization is finite-to-one,
the solution space of the linear system must be a point. We conclude that the
parametrization of the conormal variety ${\rm Con}_\lambda$ 
given in (\ref{eq:f_and_g}) has
degree $m_1!m_2!\cdots m_p!$.
\end{proof}

At this point it is important to note that our approach in this
section is quite classical. All the ingredients are known and
have been published elsewhere. What is new is their arrangement
and interpretation. For instance, the parametrization (\ref{eq:f_and_g})
appears in \cite{IK} but the connection to dual and conormal varieties
was not made explicit. Likewise, our Lemma~\ref{lem:eins}
is among the statements in \cite[Theorem 7.1]{Kat}.  The dual variety 
does make an appearance in \cite[Section 6]{Kat} but it was not seen
 that the dual of a multiple loot locus for a hook is again such a locus.
By putting all the known puzzle pieces together, we can go further in this paper.
For instance, Katz states an inequality in \cite[Corollary 6.4]{Kat}. He conjectures
in his introduction \cite[p.~220]{Kat} that equality holds.
The following result  proves Katz' conjecture.

\begin{corollary} The variety of binary forms of degree $n$ with a root of multiplicity $a$ 
satisfies
$$\deg \bigl((\Delta_{\{1^{n-a},a \} })^\vee \bigr) \,\,=
\,\, \deg(\Delta_{ \{ 1^{n-a+1},a-1 \} }) \,\, = \,\,
(a-1)(n-a+2). $$
\end{corollary}

\begin{proof}
The dual on the left is $\Delta_{ \{1^{a-2},n-a+2 \} }$ by Corollary~\ref{cor:selfdual}.
Using Hilbert's formula  \eqref{eq:hilbert},
we see that both varieties have the same degree, namely $\,(a-1)(n-a+2)$.
\end{proof}

We close this section by recording the dimension of the dual variety,
and by pointing out that our parametrization is in fact always {\em identifiable},
i.e.~birational modulo permutations.

\begin{theorem} \label{thm:neun}
For any partition $\lambda$ of $n$, the
dual variety $(\Delta_\lambda)^\vee$ has dimension~$n{-}m_1{-}1$.
The generalized additive decomposition in (\ref{eq:f_and_g})
represents a  finite-to-one parametrization 
$\, \mathcal{J}_\lambda \dashrightarrow (\Delta_{\lambda})^{\vee} \subset \PP^n$.
This rational map has degree  $m_2!\cdots m_p!$ and it is given~by
\begin{equation}
\label{eq:dualpara} 
\biggr(\bigl((t_i:s_i))_{\lambda_i \not=1} \bigr) \,,\, (g_1,\ldots,g_d) \biggr) \,\,\mapsto \,\,\,
g(u,v) \, \,=\, \sum_{i=1 \atop \lambda_i \neq 1}^d (s_i u + t_i v)^{n-\lambda_i+2} \cdot g_i(u,v). \quad
\end{equation}
\end{theorem}

\begin{proof} It was shown in \cite[Corollary 7.3]{Kat} that $
\dim((\Delta_{\lambda})^{\vee}) =n-m_1-1$. This is also the dimension of
the parameter space $\mathcal{J}_\lambda$.
Theorem \ref{main2} implies that the map (\ref{eq:dualpara}) is finite-to-one.

Fix a generic point $g(u,v)$ in
the dual variety $(\Delta_\lambda)^\vee$. Consider
any smooth point $f(x,y)$ of the primal $\Delta_\lambda$ at which $g$ is tangent.
The unordered set of linear forms $\ell_i$ that appear with multiplicity $\geq 2$
in $f(x,y)$ can be recovered uniquely by Lemma~\ref{lem:zwei}. Permuting
linear forms that appear with the same multiplicity accounts for the size
$m_2!\cdots m_p!$ of the fiber when writing out the multipliers 
$(s_i u + t_iv)^{n-\lambda_i+2}$ in (\ref{eq:dualpara}).
At this point, we still need to recover the forms
$(g_1,\ldots,g_d)$, but this can be done uniquely by the
same argument as in the proof of Theorem~\ref{main2}. We conclude 
that the degree of the map (\ref{eq:dualpara}) equals $m_2 ! \cdots m_p !$
\end{proof}

\begin{example} 
 \label{ex:notmorphism} \rm
 The rational parametrization in Theorem \ref{thm:neun}     can have base points,
 so it is generally not a morphism. Let $\lambda = (3,2)$, so 
 $(\Delta_\lambda)^\vee$ is the  hypersurface in $\PP^5$
 referred to as {\em little apple} in Example~\ref{ex:apple}.
 The toric fourfold $\mathcal{J}_\lambda$ is the free
 join of $\PP^1 \times \PP^1$ and $\PP^1 \times \PP^0$.
 In these products,   we fix the points $\bigl((s_1:t_1),g_1\bigr)$ 
 and $\bigl((s_2:t_2),g_2 \bigr)$,
 where  $ s_1 = s_2 = t_1 = t_2 = 1$,  $ g_1 = u+v$, and $g_2 =  -1$. 
 We also fix the scalar $1$ for a point on the line that joins them. These choices specify
 a point in $\mathcal{J}_\lambda$. Plugging into the
 formula  (\ref{eq:dualpara}), we obtain
 $$ g(u,v) \, =\,  (s_1 u + t_1 v)^4 \cdot g_1 \, + \,
 (s_2 u + t_2 v)^5 \cdot g_2 \,= \,  
  \,\, (u+v)^4 \cdot (u+v) \,+\, (u+v)^5 \cdot (-1)  \,= \, 0 . $$
  Hence our point    in $\mathcal{J}_\lambda$
  is a base point of the parametrization  (\ref{eq:dualpara})  for $\lambda = (3,2)$.
  \hfill $\diamondsuit$
 \end{example}
 
\section{Equations, Multidegree, and More}

Given any parametrically represented variety, it is natural
to ask for its {\em implicitization}. This concerns finding
the ideals of polynomials that vanish on ${\rm Con}_\lambda$,
on $\Delta_\lambda$, and on $(\Delta_\lambda)^\vee$.
For the multiple root loci, this is a well-studied
problem \cite{Chi1, Kur, Wey}. Before reviewing what is known,
we give the relevant definitions and we present
a census of our varieties for $n \leq 7$.

\begin{table}
\begin{center}
$$
\begin{array}{|c|c|c|c|c||c|}
\hline
\lambda & \text{Eqns of $\Delta_\lambda$} & \text{Multidegree} & \text{Eqns
of $(\Delta_\lambda)^\vee$} & \text{Hooks} &\text{ED-degrees}\\
\hline
\hline
2 & 2 & 2,2 & 2 & 2 & 2, 4 \\
\hline
3 & 2^3 & 0,3,4 & 4 & 21 & 3, 7 \\
21 & 4 & 4,3,0 & 2^3 & 3 & 3, 7 \\
\hline
4 & 2^6 & 0,0,4,6 & 6 & 211& 4, 10 \\
31 & 2^1,3^1 & 0,6,6,0 & 2^1,3^1 & 31 & 4, 12 \\
211 & 6 & 6,4,0,0 & 2^6 & 4 & 4, 10 \\
22 & 3^7 & 0,4,6,3 & 3 & 4,4 & 7, 13 \\
\hline
5 & 2^{10} & 0,0,0,5,8 & 8 & 2111 & 5, 13 \\
41 & 2^3 & 0,0,8,9,0 & 4^6 & 311 & 5, 17 \\
311 & 4^6 & 0,9,8,0,0 & 2^3 & 41 & 5, 17 \\
2111 & 8 & 8,5,0,0,0 & 2^{10} & 5 & 5, 13 \\
221 & 5^{10} & 0,12,16,6,0 & 3^4 & 5,5 & 16, 34 \\
32 & 4^{28} & 0,0,12,21,12 & 12 & 41,5 & 21, 45 \\
\hline
6 & 2^{15} & 0,0,0,0,6,10 & 10 & 21111 & 6, 16 \\
51 & 2^6 & 0,0,0,10,12,0 & 4^1,5^3,6^1 & 3111 & 6, 22 \\
411 & 2^1,3^3,4^1 & 0,0,12,12,0,0 & 2^1,3^3,4^1 & 411 & 6, 24 \\
3111 & 4^1,5^3,6^1 & 0,12,10,0,0,0 & 2^6 & 51 & 6, 22 \\
21111 & 10 & 10,6,0,0,0,0 & 2^{15} & 6 & 6 , 16 \\
2211 & 7^{13} & 0,24,30,10,0,0 & 3^{10} & 6,6 & 28, 64 \\
222 & 4^{45} & 0,0,8,16,12,4 & 4 & 6,6,6 & 20, 40 \\
33 & 3^{29} & 0,0,0,9,18,12 & 12 & 51,51 & 19, 39 \\
321 & 4^1,5^3,6^{31} & 0,0,36,56,24,0 & 4^1,6^1 & 51,6 & 44, 116 \\
42 & 2^1,3^3,4^{31} & 0,0,0,16,30,18 & 18 & 411,6 & 26, 64 \\
\hline
7 & 2^{21} & 0,0,0,0,0,7,12 & 12 & 211111 & 7, 19\\
61 & 2^{10} & 0,0,0,0,12,15,0 & 6^{10} & 31111 & 7, 27 \\
511 & 2^3,3^4 & 0,0,0,15,16,0,0 & 4^{20} & 4111 & 7, 31 \\
4111 & 4^{20} & 0,0,16,15,0,0,0 & 2^3,3^4 & 511 & 7, 31\\
31111 & 6^{10} & 0,15,12,0,0,0,0 & 2^{10} & 61 & 7, 27\\
211111 & 12 & 12,7,0,0,0,0,0 & 2^{21} & 7 & 7, 19 \\
22111 & 9^{16}  & 0,40,48,15,0,0,0 & 3^{20} & 7,7 & 43, 103 \\
2221 & 6^{78} & 0,0,32,60,40,10,0 & 4^6 & 7,7,7 & 62, 142 \\
3211 & 6^{10}, 8^{38} & 0,0,72,105,40,0,0 & 4^5, 6^{10}& 61,7 & 73, 217  \\
322 & 6^{364} & 0,0,0,36,80,66,24 & 24 & 61,7,7 & 94, 206 \\
331 & 3^{10} & 0,0,0,27,48,24,0 & 7^8 & 61,61 & 39, 99 \\
421 & 4^{20}, 6^{42}& 0,0,0,48,80,36,0 & 8^1,10^3,11^2,12^{49} & 511,7 & 52, 164 \\
43 & 3^{10}, 4^{66} & 0,0,0,0,24,51,36 & 36 & 511,61 & 51, 111  \\
52 & 2^3,3^4,4^{38} & 0,0,0,0,20,39,24 & 24 & 4111,7 & 31, 83 \\
\hline
\end{array}
$$
\caption{Multiple root loci, their duals, and their conormal varieties for $n \leq 7$. 
\label{tab:eins}  }
\end{center}
\end{table}

We write $I({\rm Con}_{\lambda})$ for the 
ideal of the conormal variety
in the $\mathbb{N}^2$-graded polynomial ring
$$ \QQ[{\bf a},{\bf b}] = 
\QQ[a_0,a_1,\ldots,a_n, b_0,b_1,\ldots,b_n], \quad {\rm with}
\quad {\rm deg}(a_i) = (1,0), \,{\rm deg}(b_i) = (0,1). $$
The ideal $I({\rm Con}_\lambda)$ is bihomogeneous and prime.
According to textbook definition \cite[\S 8.5]{MS}, the  {\em multidegree}  
of the $\mathbb{N}^2$-graded algebra
 $\QQ[{\bf a},{\bf b}]/I({\rm Con}_\lambda)$
is a binary form of degree $n+1$:
\begin{equation}
\label{eq:multidegree}
 \mathcal{C}_\lambda (t_1,t_2) \,= \, 
\delta_1 t_1 t_2^n+ 
\delta_2 t_1^2 t_2^{n-1}+ 
\cdots + \delta_{n} t_1^n t_2.
\end{equation}
The coefficients $\delta_i$ are nonnegative integers, 
known among geometers as the {\em polar classes} of
the variety $\Delta_\lambda$.
Geometrically, $\delta_i$ is the number of intersection points
in $(L_i \times M_{i}) \cap {\rm Con}_\lambda$
where $L_i$ is a generic plane of dimension $i$ in $\PP(V)$
and $M_{i}$ is a generic plane of codimension $i{-}1$ in $\PP(V^\vee)$.
In particular, we have $\delta_i = 0$ for $i < {\rm codim}(\Delta_\lambda)$
and for $i > {\rm dim}\bigl((\Delta_\lambda)^\vee)+1$.

The ideal of the multiple root locus $\Delta_\lambda$
is obtained by eliminating the variables $b_0,\ldots,b_n$,
and the ideal of its dual $(\Delta_\lambda)^\vee$ 
is obtained by eliminating the variables $a_0,\ldots,a_n$.
In symbols, 
\begin{equation}
\label{eq:elimideal}
I(\Delta_\lambda) \,=\, I({\rm Con}_\lambda) \,\cap\, \mathbb{Q}[{\bf a}]
\quad \hbox{and} \quad
I\bigl((\Delta_\lambda)^\vee\bigr) \,=\, I({\rm Con}_\lambda)\, \cap\, \mathbb{Q}[{\bf b}].
\end{equation}
These elimination ideals are $\mathbb{N}$-graded and prime.
Their degrees and codimensions in the respective $\PP^n$
can be read off from the first term and the last term of the multidegree:
$$ \mathcal{C}_\lambda(t_1,t_2) \, = \,
{\rm deg}(\Delta_\lambda) \cdot t_1^{{\rm codim}(\Delta_\lambda)} t_2^{{\rm dim}(\Delta_\lambda)+1} \,+\,\cdots \,+\, {\rm deg}((\Delta_\lambda)^\vee)  \cdot
  t_1^{{\rm dim}((\Delta_\lambda)^\vee)+1}  t_2^{{\rm codim}((\Delta_\lambda)^\vee)} . $$

We computed the objects in (\ref{eq:multidegree}) and (\ref{eq:elimideal})
for all partitions  up to $n = 7$. Our results are listed in
Table \ref{tab:eins}. Each row corresponds to one partition $\lambda$.
The second column lists the degrees of the minimal generators
of $I(\Delta_\lambda)$. The third column
lists the polar classes $\delta_1,\delta_2,\ldots,\delta_n$.
The leftmost nonzero entry in that list is the degree of
$\Delta_\lambda$ and the rightmost nonzero entry is
the degree of $(\Delta_\lambda)^\vee$. The codimension
of the variety in question is the number of consecutive left (resp.~right) zeros plus one.
The fourth column lists the degrees of the minimal generators
of $I\bigl((\Delta_\lambda)^\vee\bigr)$.
The fifth column lists the corresponding collection of hooks 
$ \{1^{\lambda_i-2}\!,\,n-\lambda_i+2\}$ that make up
the dual variety in the join construction of Corollary~\ref{main3}.

Finally, the last column refers to the Euclidean distance degrees
in two coordinate systems. This will be explained
in Section 5. The second one is always
$\mathcal{C}_\lambda(1,1) = \delta_1 + \delta_2 + \cdots + \delta_n$.

We now discuss some general facts that a reader might discover
by  looking at Table~\ref{tab:eins}.  Whenever all hooks 
dual to $\lambda$ are identical then $(\Delta_\lambda)^\vee$ is
a secant variety of that hook variety.

\begin{proposition}
\label{prop:sigmak}
Suppose $n \geq k(n-a+2)$. The $k$-th secant variety of the variety of
 binary forms of degree $n$ with a root of multiplicity $a$
is projectively dual to the multiple root locus with partition
$\lambda = \{ 1^{n-k(n-a+2)}, (n-a+2)^k\}$. In symbols,
\begin{equation}
\label{eq:sigmak}
\sigma_k\left(\Delta_{\{1^{n-a}, a\}}\right) \, = \, 
(\Delta_\lambda)^\vee.
\end{equation}
This secant variety is non-defective: it 
has the expected dimension $k(n-a+1)+k-1$.
\end{proposition}

\begin{proof}
This is the special case of Corollary~\ref{main3}
where $\lambda$ has the shape of a rectangle together with
some singleton blocks, say $\lambda =  \{1^{n-ku} , u^k\}$,
and we set $a = n-u+2$. The statement  concerning the dimension
follows from the first sentence in Theorem~\ref{thm:neun}.
\end{proof}

The best known special case of this statement arises when $a=n$, or $u=2$.
The dual of $\Delta_{\{1^{n-2k},2^k\}}$ is the $k$-th secant variety
of the rational normal curve $\Delta_{(n)}$. Its ideal is generated in
degree $k+1$, namely by the minors of a Hankel matrix. We see this
in Table~\ref{tab:eins} for $\lambda = 21,211,22,2111,221,\ldots$.
Also of interest is the case of a rectangular partition, when the
  secant variety  (\ref{eq:sigmak}) is a hypersurface. Oeding's
  formula (\ref{eq:oeding}) and Proposition~\ref{prop:sigmak} imply
  
\begin{corollary}
The secant variety referred to in Proposition \ref{prop:sigmak},
under the same hypothesis,
is a hypersurface if and only if $n = k(n-a+2)$. The degree of that hypersurface~is
$$ 
{\rm deg} \bigl(\,\sigma_k(\,\Delta_{\{1^{n-a}, \,a\}}\,)\, \bigr) \,\, = \,\,\, (k+1) (n-a+1)^k.
$$
\end{corollary}

\begin{example} \rm
Let $k=2$, $n=6$, $a = 5$, so $\lambda = (3,3)$.
The variety (\ref{eq:sigmak}) is a hypersurface of degree $12$ in $\PP^6$.
It consists of all binary  sextics $\, f =  \ell_1^5 \ell_2 + \ell_3^5 \ell_4 $
where the $\ell_i$ are linear forms. 

Consider also the case $n=12$. Here the construction gives four interesting hypersurfaces:
$$ \begin{matrix} 
k & a  &  \lambda &  {\rm forms} & {\rm degree} \\
2 &   8  &     (6,6)       &       \ell_1^ 8 g_1 + \ell_2^8 g_2            &        75             \\
3 & 10  &      (4,4,4)       &  \ell_1^{10} g_1 + \ell_2^{10} g_2 + \ell_3^{10} g_3        &  108              \\
4 & 11  & (3,3,3,3)  &  \ell_1^{11} g_1 + \ell_2^{11} g_2 + \ell_3^{11} g_3 + \ell_4^{11} g_4   & 80  \\
6 & 12    &     (2,2,2,2,2,2)     &   \ell_1^{12}  + \ell_2^{12} + \cdots + \ell_6^{12}         &      7  \\
\end{matrix}
$$
Of course, the last hypersurface is defined by the determinant of a  $7 \times 7$ Hankel matrix.
\hfill $\diamondsuit$
\end{example}

The containment relation among multiple root loci is
the order on partitions by refinement. Indeed,
$\Delta_\lambda \subset \Delta_\mu$ if and only if
$\mu$ {\em refines} $\lambda$, i.e.~every part of $\lambda$ is
a sum of parts of~$\mu$. Our next result characterizes
the containment relation among their dual varieties.
Note that
\begin{equation}
\label{eq:dualcontainment}
(\Delta_\lambda)^\vee \subset  (\Delta_\mu)^\vee
\end{equation}
cannot hold unless $\lambda$ has more $1$'s than $\mu$,
for dimension reasons.
Given a partition $\lambda = \{ 1^{m_1}, 2^{m_2}, \ldots, p^{m_p} \}$,
we write $\lambda'$ for the partition $\{1^{m_2} ,\ldots ,(p-1)^{m_p}\}$.
Note that if $\lambda$ is a partition of $n$ then 
$\lambda'$ is a partition of $|\lambda'| = n-d$, where $d=\sum m_i$ is the number of parts.

\begin{proposition}
Given two partitions $\lambda$ and $\mu$ of $n$, the inclusion (\ref{eq:dualcontainment}) holds
if and only if $|\lambda'| \leq |\mu'|$
and, by adding to the parts, $\lambda'$ can be transformed to a partition $\lambda''$
refined by $\mu'$.
\end{proposition}

\begin{proof}
This can be seen from Corollary \ref{main1}. Let $\ell^{\lambda'}$
be the differential operator that is used in that corollary
to characterize $(\Delta_{\lambda'})^\vee$. Here $\ell$ represents
an arbitrary collection of linear forms.
Our refinement condition means that every form $g$ annihilated by $\ell^{\lambda'}$
is also annihilated by $\ell^{\mu'}$. This condition is  equivalent
to (\ref{eq:dualcontainment}).
More precisely, write $m_a$ for the number of parts in $\lambda'$, 
and $m_b$ for that in $\mu'$. If such a $\lambda''$ exists,  then we claim that
$$
\bigcup_{\ell_{m_a}}{\rm Ann}((\ell_{m_a})^{\lambda'})
\,\,\subset\,\,
\bigcup_{\ell_{m_a}}{\rm Ann}((\ell_{m_a})^{\lambda''})
\,\,\subset\,\,
\bigcup_{\ell_{m_b}}{\rm Ann}((\ell_{m_b})^{\mu'}).
$$
Here, $\ell_k$ is a $k$-vector of linear forms and the exponent is written in multi-index notation. The first inclusion holds since each $(\ell_{m_a})^{\lambda''}$ can be 
seen as a multiple of $(\ell_{m_a})^{\lambda'}$, and the second 
inclusion holds since being annihilated by $(\ell_{m_a})^{\lambda''}$ can 
be seen as a special case of being annihilated by $(\ell_{m_b})^{\mu'}$, 
where $\ell_{m_b}$ is obtained by duplicating some parts of $\ell_{m_a}$.
\end{proof}

\begin{example} \rm
We can check the inclusion (\ref{eq:dualcontainment}) for various cases in Table \ref{tab:eins}.
The inclusion holds for $\lambda = 2211$ and $\mu = 321$. Indeed,
$\lambda' = 11$, and  $\mu' = 21$ refines $\lambda''  = 21$.
 Likewise, if $\lambda = 321$ and $\mu = 222$ then
   $\mu'=111$ refines $\lambda' = \lambda'' = 21$.
   This explains the inclusions
 $(\Delta_{2211})^\vee \subset (\Delta_{321})^\vee \subset (\Delta_{222})^\vee \subset \PP^6$.
 It follows that the unique quartic polynomial $4^1$ that vanishes on
$ (\Delta_{321})^\vee $ must be the $4 \times 4$ Hankel determinant 
whose hypersurface is  $(\Delta_{222})^\vee$.
\hfill $\diamondsuit$
\end{example}

We now discuss some of the entries in the column {\em Eqns of $\Delta_\lambda$}
in Table \ref{tab:eins}, and thereby review the literature on  the ideals $I(\Delta_\lambda)$.
For $\lambda = (n)$,  we see the 
$\binom{n}{2}$ quadrics that define the rational number curve,
and for $\lambda = \{1^{n-2},2\}$ we get the discriminant of degree $2n-2$.
If $a \geq \lfloor n/2 \rfloor+ 2$ then $I(\Delta_{\{1^{n-a},a\}})$ is
generated in degree $\leq 4$, as shown by Weyman \cite{Wey}.
The case $a = \lfloor n/2 \rfloor + 1$ is of special interest:
here $\Delta_{\{1^{n-a},a\}}$ is the {\em nullcone},
i.e.~the variety defined set-theoretically by all ${\rm SL}_2$-invariants of binary  $n$-ics.
Note that the nullcone is self-dual when $n$ is even
(by Corollary~\ref{cor:selfdual}), and its ideal is generated
by quartics when $n$ is odd.

 Chipalkatti \cite[Conjecture 6.1]{Chi1} had conjectured
that $I(\Delta_\lambda)$ is generated in degree $\leq 4$ whenever
$\lambda$ has only  $d =2$ parts, and this was proved by
Abdesselam and Chipalkatti in \cite[Proposition 20]{AC2}.
For further details on their approach see \cite[Section 7]{AC1}.
The papers \cite{AC1, AC2, Chi1, Wey} stress the fact that all our ideals are
invariant under the action of ${\rm SL_2}$, and hence each of their
graded components are direct sums of irreducible ${\rm SL}_2$-modules.
Chipalkatti describes the minimal free resolutions of
$I(\Delta_{\lambda})$ in terms of ${\rm SL}_2$-modules 
for $\lambda = (3,2)$  in \cite[\S 4.1]{Chi1}
and for $\lambda = (3,3)$ in \cite[\S 4.2]{Chi1}.
He discusses  the
$364$ sextics for $\lambda = (3,3,2)$ in  \cite[\S 3.1]{Chi1}.

The ideal $I((\Delta_\lambda)^\vee)$ of the dual variety
has been studied only in two cases, namely
when $\lambda$ is a hook (by self-duality) and when
$\lambda$ has only parts $1$ and $2$. In the
latter case, it is generated by minors of a Hankel matrix.
In all other cases, little seems to be known.
Of course,  $I((\Delta_\lambda)^\vee)$ 
will often be principal (when $m_1 = 0$), and we seek to
find the generator of such an ideal explicitly.
This is accomplished for some interesting cases in Section 4.

\bigskip \bigskip

\section{Real Rank Boundaries}

In this section we present an application of our duality theory 
to the study of real ranks of binary forms \cite{Ble, CO}.
Let $f \in \RR[x,y]_n$ be a general binary form of degree $n$
with real coefficients. The complex rank of $f$ equals
$r=  \lceil (n+1)/2 \rceil$. This means that $f$ is a sum of $r$
powers of linear forms over $\CC$, but not fewer.
We consider the set $\mathcal{R}_n$ 
of all real binary forms $f$ that admit a rank $r$ decomposition also
over $\RR$. In other words, $\mathcal{R}_n$ is the set
of all forms $f \in \RR[x,y]_n$ such that both the real rank of $f$ 
and the complex rank of $f$ are equal to~$r$.

The set $\mathcal{R}_n$ is a full-dimensional
semi-algebraic set inside $\RR[x,y]_n$. 
Its {\em topological boundary} $\partial \mathcal{R}_n$ is 
the set-theoretic difference of the closure of $\mathcal{R}_n$
minus the interior of the closure of $\mathcal{R}_n$.
Thus, if $ f \in \partial \mathcal{R}_n$ 
then every open neighborhood of $f$ 
contains a generic form of real rank equal to  $r$
and also a generic form of real rank bigger than $r$.
The topological boundary $\partial \mathcal{R}_n$
is a semi-algebraic subset of pure codimension one
inside the real affine space $\RR[x,y]_n$. 

We define the {\em real rank boundary}, denoted
$\partial_{\rm alg}(\mathcal{R}_n)$, to be the Zariski closure
of the topological boundary $\partial \mathcal{R}_n$. We view
$\partial_{\rm alg}(\mathcal{R}_n)$ as a
closed subvariety  in the complex projective space $\PP(\CC[x,y]_n) = \PP^n$.
It is pure codimension one, so it is defined by a unique
(up to scaling) squarefree polynomial in
the coordinates $a_0,a_1,\ldots,a_n$ on $\PP^n$.
We compute that polynomial:

\begin{theorem} \label{thm:realrank}
Let $n \geq 5$.
If $n = 2k-1$ is odd then the real rank boundary $\partial_{\rm alg}(\mathcal{R}_n)$ is
an irreducible hypersurface of degree $2k(k-1)$ in $\PP^n$. This hypersurface is
the dual $(\Delta_\lambda)^\vee$ of the
multiple root locus $\Delta_\lambda$ where $\lambda = \{2^{k-2},  3\}$. Its
defining polynomial is the discriminant of 
\begin{equation}
\label{eq:qpoly}
 q(u,v)\, \,\, = \,\,\, {\rm det} \begin{pmatrix}
\, u^k & u^{k-1} v & \cdots & u v^{k-1} & v^k \\
\, a_0 & a_1 & \cdots & a_{k-1} & a_k \\
\, a_1 & a_2 & \cdots & a_{k} & a_{k+1} \\
\, a_2 & a_3 & \cdots & a_{k+1} & a_{k+2} \\
\, \vdots  & \vdots & \ddots & \vdots & \vdots \\
\, a_{k-1} & a_{k} & \cdots  & a_{n-1} & a_n 
\end{pmatrix} .
\end{equation}
If $n = 2k$ is even then the real rank boundary $\partial_{\rm alg}(\mathcal{R}_n)$ is
the union of the two irreducible hypersurfaces $(\Delta_\lambda)^\vee$
where $\lambda$ is $\{2^{k-3},3^2\} $ or $\{2^{k-2},4\}  $.
Their degrees are $2k(k-1)(k-2)$ and $3k(k-1)$.
These arise as irreducible
factors of the Hurwitz form of the discriminant $\Delta_{\{1^{k-1},2\}}$ 
when evaluated at the line in $\PP(\RR[x,y]_{k+1})$ that is the
 kernel of the  $k {\times} (k{+}2)$-matrix
\begin{equation}
\label{eq:3by5matrix}
\begin{pmatrix}
a_0 & a_1 & \cdots & a_{k} & a_{k+1} \\
a_1 & a_2 & \cdots & a_{k+1} & a_{k+2} \\
\vdots  & \vdots & \ddots & \vdots & \vdots \\
a_{k-1} & a_{k} & \cdots  & a_{n-1} & a_n 
\end{pmatrix} .
\end{equation}
A third irreducible factor appearing in this specialized Hurwitz form is
the Hankel determinant $(\Delta_{\{2^k\}})^\vee$, 
which has degree $k+1$, but this
 is not part of the real rank boundary  $\partial_{\rm alg}(\mathcal{R}_n)$.
\end{theorem}

For the relevant background on Chow forms and Hurwitz forms
we refer to \cite{GKZ} and \cite{Stu}.
Before presenting the proof of Theorem \ref{thm:realrank}, 
here is an illustration of the first three cases.
 
\begin{example}[Little Apple and Big Apple] \label{ex:apple} \rm
Let $n=5$, so $k=3$ and $\lambda = (3,2)$.
In this first case, Theorem~\ref{thm:realrank} was proved by
Comon and Ottaviani in \cite[\S 5]{CO}.
The polynomial defining the hypersurface  $(\Delta_\lambda)^{\vee} \subset \PP^5$
is their {\em apple invariant} $I_{12}$. It has
 $228$ terms of degree $12$.
 
 The next odd case is $n=7$, so $k=4$ and $ \lambda = (3,2,2)$.
 The real rank boundary  $(\Delta_\lambda)^{\vee}$  is~a hypersurface of
 degree $24$ in $ \PP^7$, namely the join of the surface
 $\Delta_{(6,1)}$ and the threefold $\sigma_2(\Delta_{(7)}) = \sigma_2 (\nu_7(\PP^1))$.
 Its defining polynomial is computed  via (\ref{eq:qpoly}). It  has $38082$ terms.
 \hfill $\diamondsuit$
\end{example}

\begin{example}[Chow and Hurwitz forms]
\label{ex:itedis} \rm
Let $n=6$. We consider the discriminant $\Delta_{(2,1,1)} $ of
binary forms of degree $k+1 = 4$. This discriminant is a threefold of
degree $6$ in $\PP^4$. Its singular 
 locus consists of the surfaces $\Delta_{(3,1)}$ and $\Delta_{(2,2)}$.
The {\em Hurwitz form} of $\Delta_{(2,1,1)} $ is a polynomial
of degree $30 = 3 \cdot {\bf 6} + 2 \cdot {\bf 4} + {\bf 4}$ in the ten Pl\"ucker coordinates
on the Grassmannian of lines in $\PP^4$. 
It is reducible because $\Delta_{(2,1,1)}$ is singular in codimension $1$.
It factors~as 
\begin{equation} 
\label{eq:factor321}
{\rm Hur}(\Delta_{(2,1,1)}) \, = \,
{\rm Chow}(\Delta_{(3,1)})^3 \cdot 
{\rm Chow}(\Delta_{(2,2)})^2 \cdot {\rm Tan}.
\end{equation}
Here ${\rm Chow}(X)$ denotes the {\em Chow form} of a surface $X$ in $\PP^4$,
i.e.~the polynomial in Pl\"ucker coordinates that vanishes when
 the line intersects $X$. The degree of the Chow form
 equals the degree of $X$, so it is $6$ and $4$ in our two cases.
  The last factor ${\rm Tan}$ is the
 condition for the line to  be tangent at a smooth point of $\Delta_{(2,1,1)}$.
 This is a quartic in the Pl\"ucker coordinates.
 
 We now take our line in $\PP^4$ to be the kernel of the $ 3 \times 5$ matrix in (\ref{eq:3by5matrix}).
 More precisely, in the formula (\ref{eq:factor321})
 we substitute the $10$ maximal minors of (\ref{eq:3by5matrix})
 for the $10$ Pl\"ucker coordinates. 
 Here signs have to be taken into consideration carefully.
 Each maximal minor is a cubic, so the degrees above have to be tripled.
 Equation (\ref{eq:factor321}) now has degree 
 $90 = 3 \cdot {\bf 18} + 2 \cdot {\bf 12} + 3 \cdot {\bf 4}$.

The polynomial obtained from ${\rm Chow}(\Delta_{(3,1)})$
is $(\Delta_{(4,2)})^\vee$. It has $3140$ terms of degree~$18$.
The polynomial obtained from ${\rm Chow}(\Delta_{(2,2)})$
equals $(\Delta_{(3,3)})^\vee$. It has $560$ terms of degree~$12$.
The degree $12$ polynomial obtained from ${\rm Tan}$ is
the third power of the Hankel determinant
\begin{equation}
\label{eq:hankel44} (\Delta_{(2,2,2)})^\vee  \,\, = \,\,\, 
\begin{small}
{\rm det} \begin{pmatrix} 
a_0 & a_1 & a_2 & a_3 \\
a_1 & a_2 & a_3 & a_4 \\
a_2 & a_3 & a_4 & a_5 \\
a_3 & a_4 & a_5 & a_6 
\end{pmatrix}.
\end{small}
\end{equation}
To illustrate the statement of Theorem~\ref{thm:realrank},
we present two explicit forms that lie in $\partial \mathcal{R}_6$.

The binary form $\,  f = y^6 + 15x^4y^2 \,$ lies
in $\,\partial \mathcal{R}_6 \cap (\Delta_{(4,2)})^\vee\,$ but not in
 $\,(\Delta_{(3,3)})^\vee \cup (\Delta_{(2,2,2)})^\vee$.
Its line of apolar quartics is $\, L_f   =  \bigl\{
  s \cdot u v^3 \,+\, t \cdot  (u^4 - v^4) \,|\,   (s:t) \in \PP^1 \bigr\} $. This
  has discriminant  $  (27s^4 + 256t^4)  t^2$, with only
  real root at $ (s:t) = (1:0)$. That quartic lies in
  $\kappa = \Delta_{(3,1)}$, and it is a limit of
    quartics with four real roots, and also a limit 
  of  quartics with two real roots.
  From this we can construct generic sextics $f_{\pm \epsilon}$ close to $f$
  whose real ranks are $4$ and $5$.
  
The binary form  $\,f = y^6 + 5x^2y^4 - 5x^4y^2 - x^6\,$ lies
in $\,\partial \mathcal{R}_6 \cap (\Delta_{(3,3)})^\vee\,$ but not in
 $\,(\Delta_{(4,2)})^\vee \cup (\Delta_{(2,2,2)})^\vee$.
Its line of apolar quartics is $\, L_f   =  \bigl\{
  s \cdot (u-v)^2(u+v)^2  \,+ \,t \cdot uv(u^2+v^2)
  \,|\,   (s:t) \in \PP^1 \bigr\} $. This
  has discriminant  $    (16s^2 + t^2)^2  t^2$, with only
  real root  $ (s:t) = (1:0)$. That quartic lies in
  $\nu = \Delta_{(2,2)}$. We can
    construct nearby generic sextics $f_{\pm \epsilon}$ 
  whose real ranks are  $4$ and $5$.
  
The proof below will explain the
      relevance of $L_f$,  $\kappa$, and $\nu$ for $\partial \mathcal{R}_6$.
        It will also show why the Hankel determinant  $(\Delta_{(2,2,2)})^\vee$
        meets the boundary of $\mathcal{R}_6$ only in lower dimension.
                \hfill $\diamondsuit$
\end{example}

\begin{remark} \rm
It is instructive to see the factorization 
(\ref{eq:factor321}) using {\em iterative discriminants}.
Let $A(x)$ and $B(x)$ be univariate quartics, and consider
$A(x) + tB(x)$ where $t$ is an unknown.~Then
$$ {\rm discr}_t \bigl( {\rm discr}_x ( A(x) + t B(x) ) \bigr) \,\, = \,\,
\bigl( \,[6,6]_{A,B}\, \bigr)^3 \cdot
\bigl( \,[4,4]_{A,B} \,\bigr)^2 \cdot
\bigl( \,[4,4]_{A,B} \,\bigr) ,
$$
where $[e,e]_{A,B}$ stands for a big
expression that has degree $e$ in the coefficients
of $A$ and of $B$.
\end{remark}

\begin{proof}[Proof of Theorem \ref{thm:realrank}]
We write $f$ for a generic form in $\RR[x,y]_n$.
We first consider the odd case $n = 2k-1$.
The apolar ideal, consisting of all forms in  $\RR[u,v]$ such that
$g\bigl(\frac{\partial}{\partial x},\frac{\partial}{\partial y}\bigr)$ annihilates
$f(x,y)$, is generated by forms $q(u,v)$ of degree $k$ and $r(u,v)$ of degree $k+1$.
Furthermore, the dual form $q(u,v)$ has the Hankel determinantal representation in  (\ref{eq:qpoly}).

Since $f$ is generic, there exists
a unique decomposition $\,f(x,y) = \sum_{i=1}^k (s_i x + t_i y)^n$.
The points $(s_i:t_i)$ are the complex roots of $q$. Hence
$f$ lies in $\mathcal{R}_n$ precisely when $q$ is real-rooted.
By {\em real-rooted} we mean that $q$ is square-free
and its roots are all real.

Suppose now that the form $f$ moves and
passes through the boundary of $ \mathcal{R}_n$.
Then two real roots of $q(u,v)$ merge and become a double root.
At this point, the discriminant of $q(u,v)$ vanishes.
Corollary \ref{main1} implies that this discriminant equals $(\Delta_\lambda)^\vee$
where $\lambda = \{2^{k-2},3\}$.

Next consider the even case $n = 2k$.
The apolar ideal of $f$ is generated by two forms
of degree $k+1$. Let $L_f$ be the line in $\PP^{k+1}$
spanned by these two forms. The points $q$ on $L_f$ 
correspond to the distinct decompositions 
$\,f(x,y) = \sum_{i=1}^{k+1} (s_i x + t_i y)^n$. Namely,
the $(s_i:t_i)$ are the roots of $q$.
This means that $f$ lies in $\mathcal{R}_n$ if and only 
if some $q$ in $L_f$ is real-rooted.

The set of all real-rooted forms is a connected 
full-dimensional semi-algebraic subset of $\PP^{k+1}_\RR$,
and its algebraic boundary is the discriminant 
$\delta = \Delta_{\{1^{k-1},2\}}$.
Note that $\delta$ is a hypersurface of degree $2k$.
Its singular locus consists of the codimension
two loci $\kappa = \Delta_{\{1^{k-2},3\}}$ and
$\nu = \Delta_{\{1^{k-3},2^2\}}$.
Their degrees are ${\rm deg}(\kappa) = 3(k-1)$ and
${\rm deg}(\nu) = 2 (k-1)(k-2)$.

Suppose that the form $f$ moves along a general curve.
Consider its image under the map $f \mapsto L_f$
into the Grassmannian of lines in $\PP^{k+1}_\RR$,
here denoted ${\rm Gr}$.
As $f$ crosses the boundary of $\mathcal{R}_n$,
the line $L_f$ transitions from intersecting 
to not intersecting the subset of real-rooted forms.
At the transition point, the form $f$ is in $\partial_{\rm alg}(\mathcal{R}_n)$.
One of the following three scenarios might happen:
(i) $L_f$ intersects $\kappa$, or (ii) $L_f$ intersects $\nu$,
or (iii) $L_f$ is tangent to $\delta$ at a smooth point.
Each of these conditions defines an irreducible
hypersurface in the Grassmannian ${\rm Gr}$.
Their equations are the irreducible factors
in the Hurwitz form ${\rm Hur}(\delta)$.

The Hurwitz form of the discriminant factors as follows in the
coordinate ring of ${\rm Gr}$:
\begin{equation}
\label{eq:HurChowChow} {\rm Hur}(\delta) \,\, = \,\, {\rm Chow}(\kappa)^3 \cdot {\rm Chow}(\nu)^2 \cdot
{\rm Tan}(\delta). 
\end{equation}
The exponents $3$ and $2$ arise from the
classical Pl\"ucker formula
for the dual of a plane curve.
Now, since $\delta$ is a hypersurface of degree $2k$, its
Hurwitz form has degree 
$2k(2k-1)$.
Hence
$$ 2k(2k-1) \,\, = \,\,
3 {\rm deg}(\kappa) \,+\,
 2 {\rm deg}(\nu)\,+ \, {\rm deg}\bigl({\rm Tan}(\delta)\bigr)\,\, = \,\,
 9(k-1) \,+\,4(k-1)(k-2) \,+ \, {\rm deg}\bigl({\rm Tan}(\delta)\bigr). $$
 This uses the fact that the degree of a Chow form
 in Pl\"ucker coordinates equals the degree of its variety.
 We conclude that ${\rm Tan}(\delta)$ is a polynomial
 of degree $k+1$ in Pl\"ucker coordinates.
 
 The map $\PP^n \dashrightarrow {\rm Gr}, \,f \mapsto L_f$ is birational.
 Indeed, every generic line has the form $L_f$ for a
 unique sextic $f$ that is recovered by solving the
 differential equations represented by $L_f$.
Moreover, the base locus of this inverse map is precisely
 the tangential hypersurface ${\rm Tan}(\delta)$.
   
 We now examine what happens to the three irreducible factors
 in (\ref{eq:HurChowChow}) when pulled back under the map
$\PP^n \dashrightarrow {\rm Gr}$.
 Algebraically, the line $L_f$ is  given as the kernel of
 the matrix (\ref{eq:3by5matrix}). So, to compute the pullback
 of (\ref{eq:HurChowChow}), we need to 
 replace the Pl\"ucker coordinates by the
 (appropriately signed and scaled) maximal minors of
 (\ref{eq:3by5matrix}). These minors are polynomials
 of degree $k$, so we must multiply the above degrees by $k$.
This gives the degrees $\, 3k(k-1)$, $\,2 k(k-1)(k-2) \,$ and
$\,k(k+1)\,$ for the pullbacks of the three irreducible factors in
(\ref{eq:HurChowChow}).

The first two pullbacks are irreducible as polynomials in 
$\RR[{\bf a}] = \RR[a_0,\ldots,a_n]$. They are
\begin{equation}
\label{eq:bothchow} \begin{matrix}
 {\rm Chow}(\kappa) (L_f) \,\,= \,\,(\Delta_{ \{2^{k-2},4\}})^\vee 
 \quad \hbox{and} \quad
 {\rm Chow}(\nu) (L_f) \,\,= \,\,  (\Delta_{ \{2^{k-3},3^2\}})^\vee.
\end{matrix}
\end{equation}
As before, this follows from the description of the
  dual hypersurfaces in Corollary \ref{main1}.
  
The third factor in (\ref{eq:HurChowChow}) 
becomes a reducible polynomial in
 $\RR[{\bf a}]$. Namely, we have
\begin{equation}
\label{eq:TanDelta}
  {\rm Tan} (L_f) \,\, = \,\, \bigl( \,(\Delta_{\{2^k\}})^\vee\, \bigr)^k . 
\end{equation}
 In words: the pullback of ${\rm Tan}$ is the
 $k$th power of the Hankel determinant of order $k+1$.
 
 This can be seen as follows.
 A general point $f$ on the Hankel hypersurface satisfies  ${\rm rank}_\CC(f) = k $.
 It is annihilated by some binary form $g(u,v) $ of degree $k$, 
 and the line $L_f$ is spanned by
 $u \cdot g(u,v)$ and $v \cdot g(u,v)$. This means that
 $L_f$ is tangent to $\delta$ at $k$ points over $\CC$, given by the $k$ roots of  $g(u,v)$.
We see that the pullback  ${\rm Tan} (L_f) $
contains $(\Delta_{\{2^k\}})^\vee$ set-theoretically.
 Both are irreducible varieties, and hence they are equal as sets in $\PP^n$.
  By comparing degrees, we conclude that
 (\ref{eq:TanDelta}) holds. 
 
 Our argument also shows the following fact:
 if a line of the form $L_f$
is tangent to the discriminant $\delta$ at one smooth point
then it is tangent to $\delta$ at a scheme of length $k$.

We have proved that $\partial_{\rm alg}(\mathcal{R}_n)$ has
at most three irreducible factors. However, the correct number
is two. The two hypersurfaces that appear in 
$\partial_{\rm alg}(\mathcal{R}_n)$ are those in (\ref{eq:bothchow}).
 This was shown already for $n=6$. Towards the end of Example~\ref{ex:itedis},
 we exhibited one binary form $f$ for each of these two boundary strata.
 In general, a construction can be made as follows.
 
 We consider lines in the space $\PP^{k+1}$ of
 binary forms of degree $k+1$ that look like
 $$ \begin{matrix}
  L_\kappa  & = & \bigl\{ s \cdot u v^3 f(u,v) \,+ \, t \cdot (u^2+v^2) g(u,v) \,\,| \,\,
 (s:t) \in \PP^1 \bigr\}  \qquad \quad \hbox{and} \\
   L_\nu & = & \bigl\{ s \cdot (u-v)^2 (u+v)^2 f(u,v) \,+ \, t \cdot (u^2+v^2) g(u,v) \,\,| \,\,
 (s:t) \in \PP^1 \bigr\}  ,
 \end{matrix}
 $$
 where $f$ and $g$ are generic real-rooted binary forms
 of degrees $ k-3$ and $k-1$ respectively.
 The line $L_\kappa$ meets the discriminant $\delta$
 only in the cusp locus, and the line $L_\nu$
 meets $\delta$ only in the node locus.
 The binary form has  $k-1$ distinct real roots at these
 intersection points.  Now, for suitable choices  of $f$ and $g$, the
 discriminant of the pencil of binary forms has 
 no real roots other than $(s:t) = (1:0)$.
 We fix such $f$ and $g$, and we set
  $L =  L_\kappa$ or $L = L_\nu$.

     The pencil $L$
 consists of binary forms of degree $k+1$ that have
 precisely $k-1$ real roots.
Moreover, $L$ is not contained in the hypersurface ${\rm Tan}$.
Recall that the inverse to $ \PP^n \dashrightarrow {\rm Gr}$ is well-defined
 in a neighborbood $\mathcal{U}$ of $L$.
Hence, for each $U \in \mathcal{U}$ there exists
a unique form $f_U$ of degree $n$ such that
$L_{f_U} = U$. By construction, there exist generic points
$U_{\pm \epsilon}$ in $\mathcal{U}$ such that
$U_{+\epsilon}$ contains a real-rooted form,
and $U_{-\epsilon}$ contains no real-rooted form.
Then $f_{+\epsilon} = f_{U_{+\epsilon}}$  has real rank $k+1$ while 
$f_{-\epsilon} = f_{U_{-\epsilon}}$  has real rank $\geq k+2$.
We conclude that $f$ is in $\partial \mathcal{R}_n$.

It remains to be seen that the Hankel determinant $\Delta_{\{2^k\}}$ 
is not a factor of the real rank boundary $\partial_{\rm alg}(\mathcal{R}_n)$.
The proof is by contradiction. Suppose that
$f \in \partial \mathcal{R}_n$ has ${\rm rank}_\CC(f) = k$, and write
$f = \sum_{i=1}^k \ell_i^n$. The $\ell_i$
are linear forms over $\CC$, and these are unique. We can approximate $f$ by a sequence
of forms in $\mathcal{R}_n$. Each is a sum of $k+1$
powers of {\em real} linear forms. A convergence argument shows
that $\ell_1,\ldots,\ell_k$ must have
real coefficients as well.

Fix a generic real linear form $\ell_0$ and consider
$f + \epsilon \ell_0^n$ and $f - \epsilon \ell_0^n$ where $\epsilon > 0$ is small.
The square Hankel matrices corresponding to these two forms
are invertible. Their determinants have opposite signs. This can be seen from the
 Matrix Determinant Lemma. The two forms lie on different sides of the
Hankel hypersurface. Both have real rank $k+1$, and hence 
both lie in $\mathcal{R}_n$.
 This contradicts our assumption that $f$ is in the 
 topological boundary of $\mathcal{R}_n$.
\end{proof}

\medskip

\begin{remark} \rm
The real projective space $\PP^n_\RR$ of binary forms of degree $n$
is stratified by real rank. Let $r=  \lceil (n+1)/2 \rceil$.
Blekherman \cite{Ble} showed that each integer in $\{r,r+1,\ldots,n-1,n\}$ 
arises as the real rank of some open stratum.
It would be very interesting to determine the algebraic boundary
that separates real rank $i$ from real rank $i+1$, for any 
$i \in \{r,\ldots,n-1\}$. Currently, only the two extreme cases
are known. The algebraic boundary for $i = n-1$ is the discriminant
$\Delta_{\{1^{n-2},2\}}$, by \cite[Prop.~3.1]{CO}, and
the case $i=r$ is resolved by Theorem \ref{thm:realrank}.
\end{remark}

\section{Euclidean Distance Degrees}

This section is concerned with the following
optimization problem over the real field $K = \RR$.
Let  $h \in \RR[x,y]_n$ be a fixed binary form.
We seek $f \in \Delta_\lambda$ that is closest to $h$.
In symbols:
\begin{equation}
\label{eq:minimize}
 {\rm minimize} \,\,\langle f - h ,f -h \rangle
\,\,\hbox{ subject to} \,\, f \in \Delta_\lambda . \end{equation}
We call this the {\em special ED problem}
when the inner product is as in (\ref{eq:innerprod1})
and (\ref{eq:innerprod2}). Alternatively, we may also use
a positive definite quadratic form that is generic. The
resulting scenario  (\ref{eq:minimize})
is the {\em generic ED problem} for $\Delta_\lambda$.
For instance, if we replace (\ref{eq:innerprod2}) with
$\,\langle f, g\rangle \,= \, \sum_{i=0}^n \gamma_i a_i b_i$,
where $\gamma_0,\gamma_1,\ldots,\gamma_n$ are random positive reals,
then this leads to a generic ED problem.

The map that takes the given input $h$ to the optimal 
solution $f^* = f^*(h)$ of the problem (\ref{eq:minimize}) is an algebraic function.
The degree of that algebraic function is known as the {\em ED degree}
of the variety $\Delta_\lambda$. For an introduction to this topic
and its basic results see \cite{DHOST}. A follow-up study,
 aimed at varieties that admit a determinantal representation,
 was undertaken in~\cite{OSS}.
 
 In the algebraic approach to solving (\ref{eq:minimize}) one
 writes the critical equations using Lagrange multipliers
 and one removes the singular locus. We shall see concrete examples
 for $\Delta_\lambda$ shortly.
 The ED degree is the number of complex solutions to these
 equations for generic data $h$. Of course, the optimal
 point $f^*$ is real and it is among these complex critical points.
 
Varieties such as $\Delta_\lambda$ come with a natural invariant coordinate system, and
our special inner product (\ref{eq:innerprod1})-(\ref{eq:innerprod2})
gives the standard Euclidean distance for that coordinate system.
The corresponding ED degree is the {\em special ED degree} of $\Delta_\lambda$.
By contrast, the {\em generic ED degree} of $\Delta_\lambda$ is usually larger.
This is the degree for the generic ED problem described above.

In the last column of Table \ref{tab:eins}
we list the special ED degree and the generic ED degree.
The first pair in each box, corresponding to
the rational normal curve $\Delta_{(n)}$, equals $\,n,\,3n-2$.
These numbers were derived, for arbitrary $n$, in \cite[Example 5.12]{DHOST} 
and \cite[Corollary 8.7]{DHOST}. Both notions of ED degrees
are preserved under duality \cite[Theorem 5.2]{DHOST}, and the generic
ED degree coincides with the sum of the
polar classes, $\mathcal{C}_\lambda(1,1) = \sum \delta_i$,
by \cite[Theorem 5.4]{DHOST}.

Some of the ED degrees in Table \ref{tab:eins}
appear in \cite{OSS}. For instance, the
ED degrees $7,13$ for $\lambda = (2,2)$
appear in \cite[Example 1.1]{OSS}. When $\lambda$ 
is dual to a collections of hooks
of the form $n,n$ or $n,n,n$, the variety
$(\Delta_\lambda)^\vee$  is given by Hankel
matrices of rank $2$ or $3$. The corresponding
ED degrees in Table \ref{tab:eins} are found 
in \cite[Table 4]{OSS}. Therein, the left table for $\Lambda = \Omega_n$
gives generic ED degree, while the right table
for $\Lambda = \Theta_n$ gives special ED degree.

From Table~\ref{tab:eins} we can guess the two ED degrees for the
variety of binary forms of degree $n$ that have a root of multiplicity $a$.
Our next goal is to prove that this guess is correct.

\begin{theorem} \label{thm:itisn}
For any hook $\lambda = \{1^{n-a}, a\}$, 
the special ED degree of the variety $\Delta_\lambda$ is always $n$, independently of $a$, 
whereas the generic ED degree of $\Delta_\lambda$ equals $(2a-1)n-2(a-1)^2$.
\end{theorem}

\begin{proof}
We begin with the generic ED degree. It is the sum of the polar classes $\delta_i$.
The codimension $a-1$ of $\Delta_{\{1^{n-a},a\}}$ equals the dimension  of
its dual $(\Delta_{\{1^{n-a},a\}})^\vee =  \Delta_{\{ 1^{a-2},n-a+2 \}}$.
This means that precisely two polar classes are nonzero, and the generic
ED degree equals
\begin{equation}
\label{eq:aswritten}
 \delta_{a-1} + \delta_a \, = \,
{\rm deg}(\Delta_{ \{ 1^{n-a},a \} }) + 
{\rm deg}( \Delta_{ \{ 1^{a-2},n-a+2 \} }) \,=\,
a (n - a + 1) + (n - a + 2) (a - 1), \,
\end{equation}
by Hilbert's formula (\ref{eq:hilbert}).
This expression equals $(2a-1)n-2(a-1)^2$, as desired.

Next consider the special ED degree. Our goal is to compute the complex
critical points $f^*$ of the optimization problem (\ref{eq:minimize}) where $h$ is a 
fixed generic binary form in $V = \RR[x,y]_n$.
We now identify $V$ with its dual space $V^\vee$ by way of the
distinguished inner product (\ref{eq:innerprod1}), and we regard the
conormal variety ${\rm Con}_\lambda$
as an affine cone of dimension $n+1$ in $V \times V = \RR^{2n+2}$. 

By ED duality \cite[\S 5]{DHOST}, our 
problem is equivalent to solving the system of linear equations
\begin{equation}
\label{eq:conormal1}
 f + g = h   \quad \hbox{for} \quad (f,g) \in {\rm Con}_\lambda .
\end{equation}
These $n+1$ inhomogeneous linear equations have finitely many 
complex solutions  $(f^*,g^*)$ on  the $(n+1)$-dimensional affine 
variety ${\rm Con}_\lambda$. Their number is the special ED degree.

We now write the equations (\ref{eq:conormal1}) using 
the parametrization of ${\rm Con}_\lambda$ given in Theorem~\ref{main2}:
\begin{equation}
\label{eq:conormal2}
 h(x,y) \,\,= \,\, (t x - sy)^a \cdot g_1(x,y)\, \,+\,\, (s x + t y)^{n-a+2} \cdot g_2(x,y). 
\end{equation}
The two summand are unknown points in $\Delta_{\{ 1^{n-a},a \} }$
and in $\Delta_{\{ 1^{a-2},n-a+2 \} }$. Both are now regarded as affine
varieties in $V = \CC^{n+1}$. The forms
$g_1$ and $g_2$ are unknown and they have degrees $n-a$ and
$a-2$ respectively. Furthermore, $(s:t)$
is an unknown point in $\PP^1$.
Hence there are $n+1$ unknown parameters in total, to match the
$n+1$ given coefficients of $h(x,y)$. 
Thus (\ref{eq:conormal2})
is a square polynomial system in $\CC^{n+1}$, and we need to count its solutions.

Any representation (\ref{eq:conormal2}) over $\CC$ of the given binary 
form $h(x,y)$ will be called an {\em orthogonal decomposition  of type $a$}.
Thus, our proof reduces to establishing the following assertion:
a general binary form of degree $n$ has precisely
$n$ orthogonal decompositions of type $a$.

In what follows we describe the binary form whose zeros are the points $(s:t)$ that can occur in (\ref{eq:conormal2}). In other words, we eliminate the unknown binary  forms $g_1$ and $g_2$
from (\ref{eq:conormal2}).
Once $(s:t)$ is known, the coefficients of $g_1$ and $g_2$
can be recovered by solving a linear system of equations. So, it suffices to identify that binary
form and to show it has degree $n$. 

We define an endomorphism $L^{(k)}$ of the vector space $V = \RR[x,y]_n$ as follows:
$$ \bigl(L^{(k)} (f)\bigr)(x,y) \,\,= \,\,
 \frac{(n-k)!}{n!}\sum_{i=0}^k (-1)^i \binom{k}{i} x^{k-i} y^i \frac{\partial^k}{\partial x^i \partial y^{k-i}} f(x,y).
$$
This linear differential operator 
generalizes the ${\rm SO}_2$-invariant vector field
$$ L^{(1)} \,\, = \,\, \frac{1}{n} \biggl[
x \frac{\partial}{\partial y} \,- \, y \frac{\partial}{\partial x} \biggr] $$
and the second order operator
$$ L^{(2)} \,\, = \,\, \frac{1}{n(n-1)} \biggl[
x^2 \frac{\partial^2}{\partial y^2} \,- \,2xy 
\frac{\partial^2}{\partial x \partial y}   \,+\,
y^2 \frac{\partial^2}{\partial x^2} \biggr]. $$

The linear map $L^{(k)} : V \rightarrow V$ can be written explicitly
in terms of coordinates as follows.
If $\,f=\sum_{i=0}^n \binom{n}{i} a_i x^i y^{n-i}\,$ then the coefficients of
 $\,L^{(k)}(f) \,\,=\,\, \sum_{j=0}^n \binom{n}{j} b_j x^j y^{n-j}\,$ are
\begin{equation}
\label{eq:binomialsum}
\binom{n}{j}b_j \quad =
\sum_{i=\max\{0,k-j\}}^{n-j}(-1)^i \binom{k}{i}\binom{n-k}{i+j-k}a_{2i+j-k}.
\end{equation}
If we apply the $n$-th order operator then 
 this amounts to a rotation by $90$ degrees:
\begin{equation}
\label{eq:90degrees}
\bigl(L^{(n)}(f)\bigr)(x,y)  \,\,   = \,\, f(-y,x). 
\end{equation}
Now, the relevance of the differential operator $L^{(k)}$ for our proof is as follows:

\begin{lemma} \label{lem:keylemma}
Let $h$ be a  binary form of degree $n$.
Suppose that $(s:t) \in \PP^1$  
occurs in an  orthogonal decomposition (\ref{eq:conormal2}) of type $a$.
Then $(s:t)$ is a root of the binary form $\,L^{(a-1)}(h)$.
\end{lemma}

Lemma \ref{lem:keylemma} will be proved further below. 
We first derive the theorem from the lemma.

The linear map $L^{(a-1)}$ is not zero. Hence, for generic $h$,
the binary form $L^{(a-1)}(h)$ is nonzero and has degree $n$.
Such a binary form has at most $n$ distinct roots. Therefore,
by Lemma \ref{lem:keylemma}, we know that
the special ED degree is at most $n$.
What we must prove is that the special ED degree 
 is at least $n$. We do this by exhibiting, for every $n$ and $a$,
 one particular binary form of degree $n$
that has $n$ distinct orthogonal 
decompositions of type $a$.

We begin with the observation that the
following two binary forms have real coefficients:
\begin{equation}
\label{eq:hnxy}
h_n(x,y)\, \,\,= \,\,\,\frac{1}{2} \bigl( (x+\sqrt{-1}\cdot y)^n \, + \,(x-\sqrt{-1} \cdot y)^n \bigr), \quad 
\end{equation}
$$
k_n(x,y) \,=\, \frac{1}{2 \sqrt{-1}} \bigl(( x+\sqrt{-1}\cdot y)^n \, + \,(x-\sqrt{-1} \cdot y)^n \bigr).
$$
These two forms are invariant under rotation by $2\pi/n$, i.e.~they are fixed by the
endomorphism $\rho : V \rightarrow V$ that maps
$\,x\,$ to $\, (\cos \frac{2\pi}{n}) x - (\sin \frac{2\pi}{n}) y\,$ and
$\,y\,$ to $\, (\sin \frac{2\pi}{n}) x + (\cos \frac{2\pi}{n}) y$.

\smallskip

{\bf Case 1.} Suppose that $n$ is odd. For our special binary form we take

$$ \begin{matrix}
h_n(x,y) \,\,= \,\, x^n - \binom{n}{2}x^{n-2}y^2 + \binom{n}{4}x^{n-4}y^4 - \cdots \pm \binom{n}{n-1}xy^{n-1}.
\end{matrix}
$$
Note that all the exponents of $x$ are odd while that of $y$ are even.
For odd $a$ we have
\begin{equation}
\label{eq:hg1g2}
h \,=\, x^a g_1 + y^{n-a+2} g_2,
\end{equation}
for suitable forms $g_1$ and $g_2$ depending on $a$.
Similarly, for even $a$ we have
$$
h \,=\, y^a g_1 + x^{n-a+2} g_2.
$$
For each decomposition, we obtain $n-1$ others
by acting with the rotations  $\rho, \rho^2,\ldots,\rho^{n-1}$.

  \smallskip

{\bf Case 2.} Suppose that $n$ even and $a$ is even. We also take
$$ \begin{matrix}
h_n(x,y)\,\, = \,\, x^n - \binom{n}{2}x^{n-2}y^2 + \binom{n}{4}x^{n-4}y^4 - \cdots \pm y^n.
\end{matrix}
$$
Then (\ref{eq:hg1g2}) also holds. Now, $h_n$ is fixed by $\rho^{n/2}$,
so applying the rotations $\rho^i$ gives only $n/2$ distinct orthogonal
decompositions of type $a$.  However, we have
$h_n(x,y) = \pm h_n(y,x)$, so by permuting the two variables we
get additional decompositions. These are  not equal to any of the previous ones.
In total, this yields $\frac{n}{2} \cdot 2 = n$ distinct decompositions 
(\ref{eq:conormal2}) for $h_n$.

\smallskip

{\bf Case 3.} Suppose that $n$ even and $a$ is odd. In that case we take
$$ 
\begin{matrix}
k_n(x,y)\,\, = \,\,
 \binom{n}{1}x^{n-1}y - \binom{n}{3}x^{n-3}y^3 + \binom{n}{5}x^{n-5}y^5 - \cdots \pm \binom{n}{n-1}xy^{n-1},
 \end{matrix}
$$
and the argument is the same as in Case 2.
This completes the proof of Theorem \ref{thm:itisn}.
\end{proof}

We now turn to Lemma \ref{lem:keylemma}.
Note that this result is familiar in the case $a=2$,
when (\ref{eq:minimize}) asks for the
best rank $1$ approximation of a symmetric
$2 {\times} 2 {\times} \cdots {\times} 2$
tensor $h$.
Finding that approximation amounts to computing
the {\em eigenvectors} of $h$; see \cite[Corollary 8.7]{DHOST}.
However, by definition, the eigenvectors of $h$ are  the roots of
$\, L^{(1)}(h)  = 
(1/n) \cdot {\rm det} \begin{pmatrix}  x & y \\ \partial h/\partial x \! & \! \partial h /\partial y   \end{pmatrix} $.

\begin{proof}[Proof of Lemma \ref{lem:keylemma}]
We claim that  the operator $L^{(k)}$ satisfies the following  identity
\begin{equation}
\label{eq:importantid}
 L^{(k)}\bigl(f(x,y)\bigr) \,\,=\,\, (-1)^{n-k} \cdot L^{(n-k)} \bigl(f(-y,x)\bigr)
\qquad \hbox{for $k=0,1,\ldots,n$}.
\end{equation}
The special case $k=n$ is  (\ref{eq:90degrees}).
One can show that (\ref{eq:importantid}) holds by a direct computation, using the
formula (\ref{eq:binomialsum}). Equivalently, we check that it holds for monomials
and extend by linearity.

Suppose now that $(s:t) \in \PP^1$  
occurs in an  orthogonal decomposition (\ref{eq:conormal2}) of type $a$.
We apply the  differential operator $L^{(a-1)}$ to both sides of that equation. This implies
$$ L^{(a-1)}\bigl(h(x,y)\bigr) \,= \,
L^{(a-1)}\bigl(\,(tx-sy)^a g_1(x,y) \,\bigr) 
\,+\, L^{(a-1)}\bigl(\,(sx+ty)^{n-a+2} g_2(x,y)\,\bigr).
$$
Applying (\ref{eq:importantid}) to the summand on the right, we conclude
$$ L^{(a-1)}\bigl(h(x,y)\bigr) \,= \,
L^{(a-1)}\bigl(\,(tx-sy)^a g_1(x,y) \,\bigr) 
\,+\, (-1)^{n-a+1}L^{(n-a+1)}\bigl(\,(tx-sy)^{n-a+2} g_2(-y,x)\,\bigr).
$$
In both summands, a $k$-th order differential operator  $L^{(k)}$ 
is applied to a binary form that has $(s:t)$ as a root of multiplicity at least $k+1$.
Both of the resulting
binary forms still have $(s:t)$ among their roots.
 This shows that  $ L^{(a-1)}\bigl(h(x,y)\bigr)$
 vanishes at $(s:t)$.
\end{proof}

The difference between the two ED degrees in Theorem~\ref{thm:itisn}
is $2  (n-a + 1) (a-1)$. We shall explain this number 
and how to think about the operator $L^{(a-1)}$. If we fix
values for the parameters $s$ and $t$,   then the equation (\ref{eq:conormal2})
translates into an inhomogeneous linear system of equations
whose unknowns are the coefficients of $g_1$ and $g_2$.
We have  $n+1$ linear equations in $n=(n-a+1)+(a-1)$ unknowns.
There is  no solution for generic $s$ and $t$.

We write our inhomogeneous linear system as an $(n+1)\times (n+1)$-matrix $\mathcal{M}_{n,a}$.
The first row consists of the coefficients of $h$. The next $n-a+1$ rows
contain the monomials of degree $a$ in $(s,t)$, and the last $a-1$ rows
contain the monomials of degree $n-a+2$ in $(s,t)$. 
We seek row vectors of the form $(-1,g_1,g_2)$ that lie in the left kernel of $\mathcal{M}_{n,a}$.
The matrix has a banded structure, like
Sylvester's matrix for the resultant. For instance,  for $\lambda = (3,1)$,
$$ 
\mathcal{M}_{4,3} \,\, = \,\,
\begin{pmatrix}
  \, h_0  &    h_1  &   h_2   &    h_3  &  h_4  \\
\,  -s^3  & 3 s^2 t &   -3st^2 &   t^3  &   0  \\
\,   0  &  -s^3   & 3s^2t & -3st^2 & t^3  \\
\,  t^3 &  3 s t^2 & 3 s^2 t &     s^3   & 0  \\
\,   0  &     t^3  &  3 s t^2 & 3 s^2t & s^3 
   \end{pmatrix}.   $$
In order for our linear equations to have a solution, the determinant of
$\mathcal{M}_{n,a}$ must be zero.  

The degree of ${\rm det}(\mathcal{M}_{n,a})$ in $(s,t)$ is
precisely the generic ED degree. This is best seen from (\ref{eq:aswritten}).
If we replace  (\ref{eq:innerprod1}) by a generic inner product 
then the rows for $g_1$ change by a linear transformation, 
whilst the rows for $g_2$
change by the inverse linear transformation. 
Thus, for the  generic ED problem,
the  critical points are precisely
the roots of the determinant.

\smallskip

However, for the special ED problem, our determinant 
admits the following factorization:
\begin{equation}\label{eq:parity}
{\rm det} \bigl( \mathcal{M}_{n,a} \bigr) \,\,\, = \,\,\,
\pm \bigl(L^{(a-1)}(h)\bigr)(s,t) \cdot (s^2+t^2)^{(n-a + 1) (a-1)}.
\end{equation}    
The binary form $s^2 + t^2$ vanishes if and only if the last
$n$ rows of $\mathcal{M}_{n,a}$ are linearly dependent.
The degree of the extraneous factor is $ 2  (n-a + 1) (a-1) $,
and subtracting that from the generic ED degree gives $n$.
The remaining factor is the remarkable binary form $L^{(a-1)}(h)$.

We close this section with a conjecture, namely that the
 two ED degrees of $\Delta_\lambda$ always have the same parity.
  According to Table~\ref{tab:eins}, this holds for
   all partitions with $n \leq 7$.
   For hooks, it is proved by Theorem \ref{thm:itisn}, and the underlying reason is 
seen clearly in \eqref{eq:parity}. In general, the extraneous 
components should come from isotropic quadrics like $s^2+t^2$, 
and parallelities like $sv-tu$. These quadrics suggest that
the difference in ED degrees is even for all $\lambda$.
This conjecture is related to \cite[eqn.~(3.5)]{OSS}.
   At present we do not know how prove it.
   
\bigskip \bigskip
   
\section{ED Duality in Action}

We now illustrate how our results can be applied
to find exact solutions to the optimization problem (\ref{eq:minimize}).
Following \cite{DHOST, OSS}, our approach is to compute
all critical points and then select the best real critical point.
By ED duality \cite[\S 5]{DHOST}, the critical points are found
by solving linear equations on the conormal variety.
Given $h$, we need to compute all decompositions 
\begin{equation}
\label{eq:hfg_mantra}
 h(x,y) \,\, = \,\, f(x,y) \,+\, g(x,y)
\qquad {\rm where} \,\,\, (f,g) \in {\rm Con}_\lambda. 
\end{equation}
If $h$ is generic then the number of such decompositions is the special ED degree.
The  binary forms $f $ that arise in the 
decompositions (\ref{eq:hfg_mantra}) are precisely the
critical points on $\Delta_\lambda$
for the Euclidean distance to $h$, and similarly the
forms $g$ are the critical points on its dual $(\Delta_\lambda)^\vee$.
The proximity of the solution is reversed under duality
 because  $\,|| h-f||^2 + ||h-g||^2 =  ||h||^2$.
This follows from $h = f+g$ and  $\langle f,g \rangle = 0$. In particular, 
if $f$ is the closest point to $h$ among those on $\Delta_\lambda$,
then it is paired with the farthest critical point $g$ on $(\Delta_\lambda)^\vee$, and vice versa.

\medskip

In what follows we illustrate how one might solve
(\ref{eq:hfg_mantra}) and hence (\ref{eq:minimize})  in practice.
We begin with $n=5$ and $\lambda = (3,1,1)$. For our given data point we take the binary quintic
\begin{equation}
\label{eq:nicequinticdata}
 h(x,y) \,\, = \,\,
\frac{1}{2} \bigl(x+\sqrt{-1}\cdot y \bigr)^5 \, + \,\frac{1}{2} \bigl(x-\sqrt{-1} \cdot y \bigr)^5  \, + \, y^5 
\,\,= \, \, x^5 - 10 x^3 y^2  + 5 x y^4 + y^5. 
\end{equation}
This is a slight variant of (\ref{eq:hnxy}).
The primal problem is to find the closest quintic $f \in \Delta_{(3,1,1)}$
with a triple root, and the dual problem is to find the closest quintic
$g \in \Delta_{(4,1)}$ with a quadruple root.
By Theorem \ref{thm:itisn}, the equation (\ref{eq:hfg_mantra}) has
five solutions on ${\rm Con}_\lambda$. They are

$$
\begin{small}
\begin{array}{|c|c||c|c|}
\hline
|| g ||^2 & f  = h-g & g  = h-f & ||f||^2 \\
\hline
3.02 & (1.471 x^2-5.582xy+3.585y^2)(x+0.785y)^3 &
(-1.238x-0.735y)(0.785x - y)^4 & 13.98\\
4.92 & (-0.263x^2-0.211 xy -0.338 y^2) (x-3.132 y)^3 & 
(0.0131 x - 0.403 y)(3.132x+y)^4 & 12.08 \\
5.09 & \! (-0.263x^2+0.167xy+0.0346 y^2)(x+3.020y)^3  \! & 
(0.0152 x+0.468 y) (3.020x-y)^4 & 11.91 \\
6 & (x^2-10y^2) x^3 &  (5x+y) y^4 & 11 \\
8.12 & (1.473x^2 + 5.397xy +2.867y^2)(x-0.6732y)^3 & \!
(-2.301 x + 1.875 y)(0.6732x + y)^4 \! \! & 8.88 \\
\hline
\end{array}
\end{small}
$$
The upper left quintic $f$ in $\Delta_{(3,1,1)}$
is closest to $h$, at distance
$3.0215666805997121633^{1/2}$,
with triple root $(-0.78519451639408253233:1)$.
The lower right quintic $g$
in  $\Delta_{(4,1)}$ is closest to $h$, at distance 
$8.8808277614588859783^{1/2}$, with quadruple root
$(-1:0.67321557299682647408)$.
We could have guessed the decomposition $h=f+g$
in the fourth row from the input (\ref{eq:nicequinticdata}).
This one is indeed a critical point, but it is neither
primal optimal nor dual optimal.

\smallskip

Our five solutions to (\ref{eq:hfg_mantra})
were found by using the matrix that was introduced in Section~5:
$$ \mathcal{M}_{5,3}(s,t) \,\,=\,\,
\begin{pmatrix}
1  & 5 & 0 & -10 & 0 & 1 \\
-s^3 & 3 s^2 t & -3 s t^2 & t^3 & 0 & 0 \\
0 & -s^3 & 3 s^2 t & -3 s t^2 & t^3 & 0 \\
0 & 0 & -s^3 & 3 s^2 t & -3 s t^2 & t^3 \\
t^4 & 4 s t^3 & 6 s^2 t^2  & 4 s^3 t & s^4 & 0 \\
0 & t^4 & 4 s t^3 & 6 s^2 t^2  & 4 s^3 t & s^4
\end{pmatrix}.
$$
The triple (resp.~quadruple) roots $(s:t) \in \PP^1$ of $f$ (resp.~$g$)
 are the roots of the binary form
$$ {\rm det}\bigl(\mathcal{M}_{5,3}(s,t) \bigr) \,\,=\,\,
(s^2+t^2)^6 \cdot (s^5 - 10 s^3 t^2 - s^2 t^3 + 5 s t^4) \,\,=\,\,
-(s^2+t^2)^6 \cdot  \bigl(L^{(2)}(h)\bigr)(s,t). $$
We compute the five real roots numerically, and at each of them
we compute the left kernel of $\mathcal{M}_{5,3}(s,t)$. The result of that computation
is precisely the list of five pairs $(f,g)$  above.

This method scales well for hooks $\lambda = \{1^{n-a},a\}$.
Here, the special ED degree is always $n$,
and Lemma~\ref{lem:keylemma} furnishes the
minimal polynomial for the desired $a$-fold root of $f \in \Delta_\lambda$.
The matrix $\mathcal{M}_{n,a}(s,t)$ represents our task
 as a {\em homogeneous polynomial eigenvalue problem},
 and this makes it amenable to well-developed
 methods of numerical linear algebra; see e.g.~\cite{DT}.

For an illustration we fix $n=15$ and $a=6$.
The data point is the binary form
$$ \begin{matrix}
h(x,y) \,=\, \sum_{i=0}^{15} \binom{15}{i} u_i x^i y^{15-i} ,
\qquad \hbox{with randomly chosen coefficient vector} \qquad \qquad 
\end{matrix}
$$
$$ \qquad \qquad 
 (u_0, \dots, u_{15}) \,\,=\,\, (20,-17,3,16,12,14,-16,-5,7,8,-13,5,-13,-16,7,-11).
$$
Among forms $f$ of degree $15$ with a root of multiplicity $6$,
the following is the closest to $h$:
$$
\begin{small}
\begin{matrix}
 \sum_{i=0}^{15}\binom{15}{i}v_i x^i y^{15-i} = 
10^{-6}(x-8.70886y)^6(131903x^9{+}11375.9x^8y-\cdots - 552.901xy^8 {+}45.8419y^9) , 
\smallskip \\
(v_0,v_1,\dots,v_{15}) = (20,-17,3,16,12,14,-16,-5,7.04,8.19,-12.17,8.09,-3.78,1.42,-0.46,0.13).
\end{matrix}
\end{small}
$$
All $15$ critical points for this optimization problem are real
because $L^{(5)}(h)$ is real-rooted for our $h$.
 Its $15$ roots have the form
$(s:1)$ where $s \in \RR$. They appear in the first column~of 
$$
\begin{array}{|c|c||c|c|}
\hline
\text{root $s$ of} \,\, L^{(5)}(h) & \text{distance${}^2$} & \text{local}  \\
\hline
8.70886 & 86791 & \min \\
3.70567 & 111796 & \min \\
2.19850 & 163470 &  \max \\
0.05736 & 476068 & ? \\
-0.38870 & 550056 & ? \\
-3.49092 & 564363 & \min \\
-5.71229 & 565936 & \min \\
-0.22118 & 657621 & ? \\
1.25359 & 723240 & \max\\
0.25811 & 727831 &  ?\\
0.48187 & 774941 &  ?\\
0.80694 & 934884 &  \max\\
-0.68808 & 1058800 &  \max\\
-1.67383 & 1150260 &  \max\\
-1.06515 & 1256200 & \max\\
\hline
\end{array}
$$
By computing the left kernel of $\mathcal{M}_{15,6}(s,t)$ at each root,
we find the $15$ decompositions (\ref{eq:hfg_mantra}).
The squared norms $||g||^2$ of the dual solutions $g = h-f$
are listed in the second column.
So, the first row gives the optimal solution for 
$\lambda = \{1^{9}, 6\}$, while the last row gives
the optimal solution for the dual problem
$\mu = \{1^{4}, 11\}$. Local optima are indicated
in the third column. There are four local optima
on $\Delta_\lambda$ (marked with ``min'')
and six local optima on $\Delta_\mu$ (marked with ``max'').
These were certified using the signature of the Hessian of the distance function.
The local nature of the other five points cannot be decided 
with the second-order criterion
because their Hessians are singular, in both the primal
and the dual formulation.

\smallskip

Our optimization problem is more challenging
when $\lambda$ is not a hook. A small
interesting case is the partition $\lambda = (3,2)$.
Let us see what happens here if we take the same input $h$ as in
(\ref{eq:nicequinticdata}).
We seek a binary form $f \in \Delta_{(3,2)}$
with a triple root and a double root that is closest to $h$.
The desired decomposition (\ref{eq:hfg_mantra})
on the conormal variety ${\rm Con}_{(3,2)}$
now takes the form
$$ h(x,y) \,\, = \,\,
\alpha (tx - sy)^3 (vx - uy)^2 \,\,+\,\,  
(\beta x + \gamma y) (sx + ty)^4 \,+\,
\delta (ux + vy)^5. 
$$
This means that  the vector $(-1,\alpha,\beta,\gamma,\delta)$
lies in the left kernel of the $5 \times 6$-matrix
$$ 
\begin{small}
\begin{pmatrix} 
1  & 5 & 0 & -10 & 0 & 1 \\
-s^3u^2 & \! 2s^3uv{+}3s^2tu^2 \! &
\! -6s^2tuv{-}s^3v^2{-}3st^2u^2 \! &
\! 3s^2tv^2{+}6st^2uv{+}t^3u^2 \! & 
\! -3st^2v^2{-}2t^3uv \! & t^3v^2 \\
0 & t^4 & 4 st^3 & 6s^2t^2 & 4s^3t & s^4 \\
t^4 & 4 s t^3 & 6 s^2 t^2 & 4 s^3t & s^4 & 0 \\
v^5 & 5 uv^4 & 10 u^2 v^3 & 10 u^3 v^2 & 5 u^4 v & u^5 
\end{pmatrix}.
\end{small}
$$
 All $5 \times 5$ minors of this matrix must be $0$.
However, the ideal of $5 \times 5$ minors has some
extraneous associated primes that must be removed. This is analogous
to the factor $s^2+t^2$ in the hook case, but more complicated,
so we do not pursue that primary decomposition.

Instead we simply work directly with the  squared
Euclidean distance function
$$
D \,\,=\,\, \parallel h-\alpha(tx-sy)^3(vx-uy)^2 \parallel^2 , 
$$
and we  solve the following system of polynomial equations in five unknowns:
$$
\frac{\partial D}{\partial \alpha} \,= 
\frac{\partial D}{\partial s} = \frac{\partial D}{\partial t} \,=\,
\frac{\partial D}{\partial u} = \, \frac{\partial D}{\partial v}\,=\, 0 
\qquad \hbox{and} \quad
sv - ut \neq 0. $$
This has $20$ complex solutions, while
 the ED degree of $\Delta_{(3,2)}$ is $21$.
 One checks that  $h$ lies in the ED discriminant 
\cite[\S 7]{DHOST}. Only $4$ of the $20$ solutions are real.
Setting $t=v=1$, they are 
$$
\begin{array}{|c|c|c||c|}
\hline
\alpha & s & u & \text{distance${}^2$}\\
\hline
\phantom{-} 1.817238 & \phantom{-} 0.673272 & -1.316853 & 7.724678\\
-0.266252 & -3.020572 & \phantom{-} 0.274673 & 8.643701\\ 
-0.265424 & \phantom{-} 3.131909 & -0.387044 & 12.017703\\
\phantom{-} 1.815280 & -0.785143 & \phantom{-} 1.428712 & 13.105926\\
\hline
\end{array}
$$
The worst critical point is given in the bottom row. The corresponding quintic is
$$
f \,\,=\,\, 1.81528x^5-0.911262x^4y-5.15521x^3y^2+0.0137459x^2y^3+4.34203xy^4+1.79341y^5.
$$
ED duality gives us
the optimal solution $g = h-f$ for the dual problem on $(\Delta_{(3,2)})^{\vee}$, namely
$$ \begin{matrix}
g & \!\! = \!\! & \!\! -0.81528x^5+0.911262x^4y-4.84479x^3y^2-0.0137459x^2y^3+0.657973xy^4-0.793414y^5 \\
& \! =\! &  (\beta x+ \gamma y)(-0.785143x+y)^4 + \delta(1.428712x+y)^5 
\end{matrix}
$$
for suitable real constants $\beta, \gamma$, and $\delta$, which can be found by solving a linear system.

\smallskip

The form $g$ is very interesting for the application described in Section 4.
Recall that the generic tensor rank for binary quintics is $3$. The unique rank $3$
decomposition of the given quintic $h$ is complex. It is shown on the
left in (\ref{eq:nicequinticdata}). One might therefore ask for a
best approximation to $h$ that has real rank $3$. This question is
not well-posed because the set $\mathcal{R}_5$ from Section 4 is not closed.
Instead one should ask for the closest binary quintic 
in the closure of $\mathcal{R}_5$, i.e.~among those whose
{\em border rank} equals $3$. That optimal quintic must be our $g$,
because it is a critical point on
real rank boundary $\partial_{\rm alg}(\mathcal{R}_5)$, 
which equals the 
little apple hypersurface $(\Delta_{(3,2)})^\vee$.
The fact that $g$ has border rank $3$ is verified by the representation
$$ g \, = \,
\lim_{\epsilon \to 0}\left(\frac{1}{5\epsilon}((-0.785143+\epsilon \beta)x+(1+\epsilon \gamma)y)^5 - \frac{1}{5\epsilon}(-0.785143x+y)^5+\delta(1.428712x+y)^5 \! \right)\!. 
$$
This computation offers a concrete
illustration of Theorem \ref{thm:realrank} and  Example \ref{ex:apple}.

\medskip \bigskip \bigskip

\noindent
{\bf Acknowledgements.}\smallskip \\
We thank Michael Burr for a conversation in March 2015
that ignited our interest in the ED problem for multiple root loci.
We are grateful to Greg Blekherman for his help in getting
Theorem \ref{thm:realrank} into final shape.
This project was carried out in the summer of 2015, when  both authors visited the
National Institute of Mathematical Sciences,
Daejeon, Korea. 
Hwangrae Lee was supported by the National Research Foundation of Korea (MSIP 2011-0030044).
Bernd Sturmfels was supported by  the US National Science Foundation (DMS-1419018).

\begin{small}

\end{small}

\bigskip

\noindent
\footnotesize {\bf Authors' addresses:}

\noindent Hwangrae Lee, 
Pohang University of Science and Technology, Korea,
 {\tt meso@postech.ac.kr}

\noindent Bernd Sturmfels,
University of California, Berkeley, USA,
{\tt bernd@berkeley.edu}

\end{document}